\documentclass[uslettersize, 10pt, twocolumn]{ieeetran}

\IEEEoverridecommandlockouts                              
\usepackage{bm}
\usepackage{epsfig, microtype}
\usepackage{graphics} 
\usepackage{epsfig} 
\usepackage{times} 

\usepackage{amsmath,amsthm} 
\usepackage{amssymb}  
\usepackage{amscd}
\usepackage[noadjust]{cite}
\usepackage{float}

\newtheorem{Definition}{Definition}
\newtheorem{theorem}{Theorem}
\newtheorem{Remark}{Remark}
\newtheorem{pro}{Proposition}
\newtheorem{lem}{Lemma}
\newtheorem{cor}{Corollary}

\newcommand{\R}{\mathbb{R}}
\newcommand{\Z}{\mathbb{Z}}
\newcommand{\ol}{\overline}
\renewcommand{\cal}{\mathcal}
\renewcommand{\(}{\left (}
\renewcommand{\)}{\right )}

\newcommand{\1}{\mathbf{1}}

\newcommand{\rt}{{\rm rot}}
\newcommand{\rf}{{\rm ref}}
\newcommand{\GL}{{\rm GL}}
\renewcommand{\O}{{\rm O}}
\newcommand{\net}{{\rm Net}}
\newcommand{\0}{\mathbf{0}}

\usepackage{tikz,pgfplots}
\usetikzlibrary{arrows,chains,matrix,positioning,scopes,decorations}

\tikzset{%
  >=latex, 
  inner sep=0pt,%
  outer sep=2pt,%
  mark coordinate/.style={inner sep=0pt,outer sep=0pt,minimum size=3pt,
    fill=black,circle}%
}

\title{\LARGE \bf
Voltage Graphs and Cluster Consensus with Point Group Symmetries}

\author{Xudong Chen, M.-A. Belabbas, Tamer Ba\c sar
\thanks{Xudong Chen, M.-A. Belabbas, Tamer Ba\c sar are with the Coordinated Science Laboratory, University of Illinois at Urbana-Champaign, emails: \{xdchen, belabbas, basar1\}@illinois.edu. }}

\begin{document}

\maketitle
\thispagestyle{empty}
\pagestyle{empty}

\begin{abstract}
A cluster consensus system is a multi-agent system in which the autonomous agents communicate to form multiple clusters, with each cluster of agents asymptotically converging to the same clustering point. 
We introduce in this paper a special class of cluster consensus dynamics, termed the {\em $G$-clustering dynamics} for $G$ a point group, whereby the autonomous agents can form as many as $|G|$ clusters, and moreover,  the associated $|G|$ clustering points exhibit a geometric symmetry induced by the point group. The definition of a $G$-clustering dynamics relies on the use of the so-called voltage graph~\cite{gross1987topological}. We recall that a $G$-voltage graph is comprised of two elements---one is a directed graph (digraph), and the other is a map  assigning elements of a group~$G$ to the edges of the digraph. For example, in the case when $G = \{1, -1\}$, i.e., a cyclic group of order~$2$, a voltage graph is nothing but a signed graph. A $G$-clustering dynamics can then be viewed as a generalization of the so-called Altafini's model~\cite{altafini2012dynamics,altafini2013consensus}, which was originally defined over a signed graph, by defining the dynamics over a voltage graph. One of the main contributions of this paper is to identify a necessary and sufficient condition for the exponential convergence of a $G$-clustering dynamics. Various properties of voltage graphs that are necessary for establishing the convergence result are also investigated, some of which might be of independent interest in topological graph theory. 
\end{abstract}

\section{Introduction}
A cluster consensus (or group consensus) system is a multi-agent system in which the autonomous agents communicate to form multiple clusters, with each cluster of agents asymptotically converging to the same clustering point. 
Motivated by what is seen in nature and the hopes for manmade systems, there has been many efforts in modeling a clustering consensus system whereby local interactions among the agents can lead to a desired global behavior of the ensemble system. Often, the choice of such a model is some kind of diffusive network dynamics, possibly with a set of  external inputs injected into the evolution equations of certain individual agents that are chosen from different clusters.  We refer to~\cite{blondel2009krause,yu2010group,xia2011clustering,han2013cluster,shang2013l1} as typical examples of such cluster consensus system.                 



In this paper, we introduce a class of cluster consensus dynamics, termed the {\em $G$-clustering dynamics} for $G$ a point group, along which the $N$~autonomous agents can   form as many as $|G|$ clusters without any external input. Moreover,  the associated $|G|$ clustering points exhibit will a geometric symmetry induced by the point group~$G$. To elaborate a bit, we consider, for example, the case where $G$ is a cyclic group of order~$n$, generated by a single rotation matrix $\theta \in \R^{2\times 2}$ with $\theta ^n = I$. Then, an associated $G$-clustering dynamics partitions the agents into $n$ clusters, with the clustering points being the vertices of an~$n$-sided regular polygon. 

To introduce a $G$-clustering dynamics, we first recall the definition of a {\bf voltage graph}. In topological graph theory~\cite{gross1987topological}, a $G$-voltage graph  is defined to be a pair $(\Gamma, \rho)$, with $\Gamma = (V, E)$  a directed graph and $\rho: E\longrightarrow G$ a map from the edge set $E$ to a finite group $G$ (not necessarily a point group). The values of the map $\rho$ are said to be the {\bf voltages}, and the group $G$ is said to be the {\bf voltage group} associated with $(\Gamma, \rho)$. 
We note here that a voltage graph $(\Gamma, \rho)$ is also known as a {\it gain graph}, and this terminology is more often used in bias graph theory~\cite{zaslavsky1989biased} and matroid theory~\cite{zaslavsky1991biased}.    
We also note that in the case $G = \Z/(2) = \{1, -1\}$, i.e., the cyclic group of order~$2$, a voltage graph $(\Gamma, \rho)$ can be viewed as a signed graph~\cite{cartwright1956structural}, with $1$ and $-1$ representing the plus- and minus-sign, respectively.   
With a $G$-voltage graph at hand, we describe below the $G$-clustering dynamics. 
  
\vspace{3pt}
\noindent
{\bf The $G$-clustering dynamics.} To this end, let $\Gamma = (V, E)$ be a directed graph (or in short, {\it digraph}) of $N$ vertices, with $V = \{v_1,\ldots, v_N\}$ the set of vertices and $E$ the set of edges. We denote by $v_{i} \to v_j$ (or simply $e_{ij}$) an edge of $\Gamma$ from~$v_i$ to~$v_j$; we say that $v_i$ is an {\bf in-neighbor} of~$v_j$, and $v_j$ is an {\bf out-neighbor} of~$v_i$. For a vertex $v_i\in V$, let $\cal{N}^+(v_i)$ and $\cal{N}^-(v_i)$ be the sets of in- and out-neighbors of vertex~$v_i$, respectively.  Now, consider a multi-agent system of $N$ agents. 
Following the standard convention,  we assign to each vertex~$v_i $ of $\Gamma$ an agent $x_i \in \R^k$, and let the edges of $\Gamma$ represent the information flow. 
For a set of agents $x_1, \ldots, x_N$ in $\R^k$, set $$p:= (x_1,\ldots, x_N) \in \R^{kN}.$$ We call $p$ a {\it configuration}, and $P := \R^{kN}$ the {\it configuration space}.  
Let  $(\Gamma, \rho)$ be a voltage graph, with the voltage group~$G$ being a point group in dimension~$k$. For ease of notation, let   
$
\theta_{ij}:= \rho(e_{ij}) 
$. The {\bf $G$-clustering dynamics} of a configuration  $p = (x_1, \ldots, x_N)$ is then given by
\begin{equation}\label{eq:clusteringdynamics}
 \dot x_i = \sum_{v_j\in \cal{N}^-(v_i) }a_{ij} \, (\theta_{ij} \, x_j - x_i), \hspace{10pt} \forall i = 1,\ldots, N, 
\end{equation}
where the $a_{ij}$'s are positive constants. Note that a $G$-clustering dynamics does not require an external input. We also note that the dynamics of each agent~$x_i$ depends only on its local information, i.e., the positions of its out-neighbors $x_j$ and the associated voltages $\theta_{ij}$,  for $v_j \in \cal{N}^-(v_i)$. In particular, it does not require a global labeling of the agents that belong to the same cluster.   
\vspace{3pt}

A $G$-clustering dynamics can be viewed as a straightforward generalization of the Altafini's model~\cite{altafini2012dynamics,altafini2013consensus}; indeed,   
if each $x_i$, for $v_i \in V$, is a scalar, and $G = \{1,-1\}$, which is the (unique) nontrivial point group in dimension one, then system~\eqref{eq:clusteringdynamics} is reduced to the  standard Altafini's model. Signed graphs  have been widely used in social science: naturally the edges of a signed graph, with the labeling of plus/minus signs, can be used to model the relationships between pairs of interacting agents; specifically, an positive (resp. negative) edge of a signed graph means a friendship (resp. an antagonism) between a pair of neighboring agents. This then leads to an application of the Altafini's model in  opinion dynamics.    
Specifically, it has been shown in~\cite{altafini2013consensus} that  if $\Gamma$ is strongly connected and the associated signed graph $(\Gamma, \rho)$ is structurally balanced (the notion of structural balance is originally defined for signed graphs~\cite{cartwright1956structural},  a generalized definition for voltage graphs will be given in Subsection~\ref{ssec:preliminary}), then the $N$ scalars $x_1,\ldots, x_N$ evolve, along the dynamics~\eqref{eq:clusteringdynamics},  to form two clusters, with the pair of clustering points being the opposite of each other. On the other hand, if  $(\Gamma, \rho)$ is structurally unbalanced, then all the scalars $x_1,\ldots, x_N$ converge to zero. We further refer to~\cite{proskurnikov2014consensus,xia2015structural,Ji2015cdc} for analyses of convergence of  the  Altafini's models with time-varying network topologies.         




We extend in this paper the result about convergence of the Altafini's model to a $G$-clustering dynamics. Specifically, we assume that the underlying graph $\Gamma$ of system~\eqref{eq:clusteringdynamics} is rooted and $G$ is an arbitrary point group in dimension~$k$. We then establish a necessary and sufficient condition, in Theorem~\ref{thm:clustering}, on the $G$-voltage graph $(\Gamma, \rho)$ under which system~\eqref{eq:clusteringdynamics} is a cluster consensus system: in particular, we show that for any initial condition $p(0) = (x_1(0),\ldots, x_N(0))$, the trajectory $p(t) = (x_1(t),\ldots, x_N(t) )$ generated by system~\eqref{eq:clusteringdynamics} converges, and moreover, 
$$
\lim_{t\to \infty} x_i(t) =\theta_{ij} \lim_{t\to \infty} x_j(t),\hspace{10pt} \forall \, e_{ij} \in E.   
$$
We also establish results, in Corollary~\ref{cor:clustering}, for the problem of counting the number of clusters associated with a $G$-clustering dynamics,  and for the problem of identifying the agents that belong to the same cluster.

Of course, the proof of convergence of a $G$-clustering dynamics relies on the understanding of the underlying $G$-voltage graph. So, in the paper, we will first review some basic definitions and facts associated with a $G$-voltage graph, with $G$ an arbitrary finite group. Various properties of the $G$-voltage graph that are necessary for proving the convergence of system~\eqref{eq:clusteringdynamics} will then be established following that. 

The remainder of the paper is thus organized as follows:  Section~II is mainly devoted to the study of voltage graphs: In Subsection II-A, we recall some definitions of finite groups and directed graphs. In Subsection II-B, we review certain basic notions associated with voltage graphs---such as net voltage, structural balance, local groups, etc. Then, in Subsections II-C and II-D, we 
establish results of voltage graphs  that are necessary for the analysis of a $G$-clustering dynamics. 
Section~III is devoted to the analysis of the so-called derived graph. Roughly speaking, a derived graph of a $G$-voltage graph $(\Gamma, \rho)$ is a special covering graph of $\Gamma$, which is comprised of $|G| |V|$ vertices and $|G| |E|$ edges (a precise definition is in Definition~\ref{def:coveringgraph}). In general, a derived graph has multiple connected components. We show in Subsection III-A that any two connected components are isomorphic. Then, in Subsection III-B, we establish results about the root connectivity of each connected component, which will be of great use in the proof of convergence of a $G$-clustering dynamics. Section~IV is devoted to the analysis of a $G$-clustering dynamics. By combining the results derived in Sections~II and~III, we establish a necessary and sufficient condition for the exponential convergence of a $G$-clustering dynamics.  
Simulation results are also presented as empirical evidence of the convergence. We provide conclusions at the end of the paper.

\section{Voltage Graphs, Structural Balance, Local Groups and Their Associated Properties}

\subsection{Backgrounds of finite groups and of directed graphs}

\subsubsection{On finite groups} 
Let $G$ be a finite group, with $\mathbf{1}$ the identity element of $G$. If $G$ is comprised only of the identity element, then $G$ is said to be {\it trivial}. 
The {\bf order} of the group $G$ is its cardinality $|G|$. 
Let $H$ be a subgroup of $G$. 
It is known that the order of  $H$ divides the order of $G$; the quotient $|G|/|H|$ is the {\bf index} of $H$ in $G$. 
Let $H$ and $H'$ be two subgroups of $G$; we say that $H$ and $H'$ are {\bf conjugate} if  there is a group element $g\in G$ such that 
$
H = g\cdot H' \cdot g^{-1}  
$. 
Let $S$ be a subset of $G$; a subgroup $H$, denoted by $\langle S\rangle $, is said to be {\bf generated by} $S$ if $H$ is the smallest subgroup of $G$ that contains 
$S$. 
We further need the following definitions and notations: 

{\it a). Left- and right-cosets}. Let $H$ be a subgroup of $G$.  For a group element $g\in G$, we call $g\cdot H$ the {\bf left-coset} of $H$ with respect to $g$. For any two group elements $g_1$ and $g_2$ of $G$, the left-cosets $g_1 \cdot H$ and $g_2\cdot H$ are either disjoint or identical with each other. Thus, if we let $k:= |G|/|H|$, then  there are group elements $g_1,\ldots, g_k$ such that 
$
G = \bigsqcup^k_{i = 1} (g_i \cdot H)
$. 
We denote the collection of left-cosets of $H$ by $$G / H := \{ g_i \cdot H \mid i = 1,\ldots, k\}.$$
Similarly, for a group element $g$ and the subgroup $H$, we call $H \cdot g$ the {\bf right-coset} of $H$ with respect to $g$. There are group elements $g'_1,\ldots, g'_k$ such that   
$
G = \bigsqcup^k_{i = 1} (H\cdot g'_i)
$. 
We denote the collection of right-cosets of $H$ by
$$H\backslash G := \{H\cdot g'_i \mid i = 1,\ldots, k\}.$$\,   

{\it b). Group homomorphisms}. Let $G$ and $G'$ be two groups of the same order. A map $\tau: G\longrightarrow G'$ is said to be a {\bf group homomorphism} if for any two group elements $g_1$ and $g_2$ of $G$, we have $\tau(g_1\cdot g_2) = \tau(g_1) \cdot \tau(g_2)$. Furthermore, if $\tau$ is a bijection, then we call $\tau$ a {\bf group isomorphism}.


{\it c). Point groups}.
Let $\O(k)$ be the orthogonal group in dimension~$k$. We express $\O(k)$ as the set of $k$-by-$k$ orthogonal matrices:    
$$
\O(k) = \{\theta\in \R^{k\times k} \mid \theta^\top \theta = I \}.
$$
A group $G$ is said to be a  {\bf point group} in dimension~$k$ if it is a finite subgroup of $\O(k)$. 
Point groups are naturally  used to characterize the geometric symmetries of objects in $\R^k$. Because of the widespread relevance, point groups have been investigated extensively in the literature. In particular,  point groups in lower dimensions have been completely understood. For example, for the case $k = 1$, there is only one nontrivial subgroup of $\O(1)$,  i.e., $G = \{1, -1\}$. For the case $k = 2$, a point group $G$ is isomorphic to either $C_n$, the {\it cyclic group} of order $n$, or $D_n$ the {\it dihedral group} of order $2n$. Specifically, if $G$ is isomorphic to $C_n$, then 
$G = \langle \{\theta_{\rt,n} \}\rangle$, with  $\theta_{\rt, n}$ a rotation matrix given by
\begin{equation}\label{eq:rotationn}
 \theta_{\rt,n} := \begin{bmatrix}
 \cos(2\pi /n) & -\sin(2\pi/n) \\
 \sin(2\pi/n)  & \cos(2\pi/n)
\end{bmatrix}. 
\end{equation}
If $G$ is isomorphic to $D_n$, then 
$
G = \left\langle \{\theta_{\rt,n}, \theta_{\rf, v}\} \right \rangle 
$, with $\theta_{\rt, n}$ defined in~\eqref{eq:rotationn} and $\theta_{\rf,v}$ given by
\begin{equation}\label{eq:reflection}
\theta_{\rf, v} = 2 {vv^\top}/{\|v\|^2} - I, \hspace{10pt} \mbox{for} \hspace{5pt} v\in \R^2 - \{0\}, 
\end{equation}
which represents the reflection of the line $\{\alpha \, v \mid \alpha \in \R\}$ in $\R^2$.  Point groups in dimension three are more complicated. Roughly speaking, the isomorphism classes of point groups in dimension three fall into  fourteen  categories, seven of which are infinite families of axial groups, and the remaining seven are polyhedral groups. We refer to~\cite{coxeter2013generators} for more details. 

We note here that points groups also arose naturally in the theory of finite group representations. Specifically, let $G$ be an arbitrary finite group, and let  $\GL(k,\R)$ be the general linear group of degree~$k$, i.e.,
$$
\GL(k,\R) := \{ A \in \R^{k\times k} \mid \det A \neq 0\}. 
$$
A {\bf representation} of $G$ is a group homomorphism $h: G \longrightarrow \GL(k, \R)$. Then, it is known that $h(G)$, the image of $G$ under $h$, has to be a point group in dimension~$k$.   


\subsubsection{On directed graphs}\label{ssec:ondigraph}
A directed graph (digraph) is said to be {\it simple} if it does not contain multiple edges or self-loops. All directed graphs considered in this paper are simple.   
We introduce below some definitions and notations of simple digraphs that are needed in the paper:

{\it a). Semi-walks, -paths, and -cycles}. A {\bf semi-walk} $w$ of a digraph $\Gamma$ is an alternating  sequence of vertices  and edges:  
\begin{equation}\label{eq:semi-walk}
w = v_{i_1} \, a_1 \, v_{i_2}  \ldots v_{i_{n-1}}\, a_{n-1} \,  v_{i_n},
\end{equation} 
and for each $j = 1,\ldots, n-1$, 
either  $ a_{j} = e_{i_j i_{j+1}} $  or  $ a_{j} = e_{i_{j+1}i_j}$. 
Further, the semi-walk $w$ is said to be a {\bf walk} if  $a_j = e_{i_ji_{j+1}}$ for all $j = 1,\ldots, n-1$. If the semi-walk $w$ is comprised only of a single vertex (and hence does not contain any edge), then $w$ is said to be trivial.   The {\it length} of the semi-walk~$w$, denoted by~$l(w)$, is defined to be the number of edges contained in~$w$.  Let $l_+(w)$ and $l_-(w)$ be two non-negative integers defined as follows: 
$$
\left\{
\begin{array}{l}
l_+(w) := | \{ j \mid  a_{j} = e_{i_ji_{j+1}}, 1\le j \le n-1 \}|, \\
l_-(w) :=  | \{j \mid  a_{j} = e_{i_{j+1}i_j}, 1\le j \le n-1\}|. 
\end{array}
\right. 
$$ 
It should be clear that $l_+(w) + l_-(w) = l(w)$, and $w$ is a walk if and only if $l_-(w) = 0$. 
The semi-walk $w$ in Eq.~\eqref{eq:semi-walk} is said to be {\bf closed} if the {\it starting vertex} $v_{i_1}$ coincides with the {\it ending vertex} $v_{i_n}$.  We say that $w$ is a {\bf semi-path} if  all vertices in $w$ are pair-wise distinct, and is a {\bf semi-cycle} if  there is no repetition of vertices in $w$, other than the repetition of the starting- and ending-vertex. 
Further, we say that  $w$  is  a {\bf path} (resp. a {\bf cycle}) if $w$ is both a walk and a semi-path (resp. a semi-cycle). Note that if $w$ is a trivial semi-walk, then $w$ is also a walk, a path and a cycle. 

{\it b). Operations on semi-walks}. 
Let $\Gamma$ be a weakly connected digraph. We introduce here three types of operations on the semi-walks of $\Gamma$ that will be frequently used in the paper:


 {\it i). Concatenation of semi-walks}. 
Let $w'$ and $w''$ be semi-walks of $\Gamma$, and let the ending vertex of $w'$ coincide with the starting vertex of $w''$:   
$$
\begin{array}{l}
w' = v_{i_1}a_1 \ldots a_{n-1}v_{i_n},  \\
w'' = v_{i_n}a_{n} \ldots a_{n+m-1}v_{i_{n+m}}.
\end{array}
$$
A semi-walk $w$ is a {\bf concatenation} of $w'$ and $w''$, denoted by $w = w'\,w''$, if
$$
w = v_{i_1}a_{1} \ldots a_{n+m -1} v_{i_{n+m}}.
$$ 
Note that if $w$ is a closed semi-walk, then $w$ can be concatenated with itself. We thus denote by $w^k$ the closed semi-walk derived by concatenating $k$ copies of $w$.

 {\it ii). Inverse of a walk}. 
Let $\Gamma$ be a weakly connected voltage graph. Let $v_{i}$ and $v_{j}$ be vertices of $\Gamma$, and $w$ be a semi-walk from~$v_i$ to~$v_j$: 
$$
w = v_{i} \, a_{1}  \ldots a_{n-1} v_{j}. 
$$ 
The {\bf inverse} of $w$, denoted by $w^{-1}$, is a semi-walk from~$v_j$ to~$v_i$ derived by reversing the appearing order of vertices and edges in $w$, i.e.,
$$
w^{-1} := v_{j}\, a_{n-1}\ldots a_{1}v_{i}. 
$$\, 

{\it iii).  Cycle reduction of a closed semi-walk}.   
Let 
$
w = v_{i_1} a_{1} \ldots a_{n-1} v_{i_n}
$, with $v_{i_1} = v_{i_n}$,  be a closed semi-walk of $\Gamma$.   
Suppose that $w$ is not a semi-cycle; then, there is a vertex $v_{i_{j}}$, for $j > 1$,  such that  
$
v_{i_{j}} = v_{i_{j+k}}
$ 
for some $k > 0$.  
Let $k$ be chosen  such that it is the least positive integer for the relation above to hold. Then, the semi-walk 
$$
c_1 := v_{i_j} a_{j} \ldots  a_{j+k-1}v_{i_{j+k}}
$$ 
is a semi-cycle. We can thus derive a closed semi-walk $w_1$ by removing $c_1$ out of $w$, i.e.,  
$$
w_1:= v_{i_1}a_{1} \ldots a_{j - 1} v_{i_j} a_{j+k} \ldots a_{n-1} v_{i_n}.
$$
We call such an operation a {\bf cycle reduction} of $w$. Recall that $l(w)$ is the length of $w$.  It should be clear that $l(w_1) < l(w)$, and hence if we let 
$
w, \,  w_1 ,\,  w_2 \ldots 
$ 
be a chain of semi-walks, with each $w_i$ derived by the operation of cycle-reduction of its predecessor, then the chain must terminate in finite steps. Suppose that the chain terminates at $w_l$;  then,  $w_l$ {\it has to} be a semi-cycle. We call $w,\,  w_1,\, \ldots, \,   w_l$ the {\bf chain of cycle reductions} of $w$. 
We note here that if $w$ is a closed walk, then each $w_i$, for $i = 1,\ldots, l$, in the chain is a closed walk, and each removed semi-cycle $c_i$, for $i = 1,\ldots, l$,  is a cycle.

{\it c). Connectivities of digraphs}. We call a digraph $\Gamma$ {\bf weakly connected} if for any two vertices $v_i$ and $v_j$ of $V$, there is a semi-walk from $v_i$ to $v_j$. The digraph $\Gamma$ is said to be {\bf rooted} if there exists at least one vertex $v_r$ such that for any vertex $v_i$, there is a path from $v_i$ to $v_r$. We call  $v_r$ a {\bf root} of $\Gamma$. 
A pair of distinct vertices $(v_i,v_j)$ of $\Gamma$ is said to be {\bf mutually reachable} if there is a path from $v_i$ to $v_j$ and a path from $v_j$ to $v_i$. 
The digraph $\Gamma$ is said to be {\bf strongly connected} if each pair of distinct vertices is mutually reachable. We also note that if $\Gamma$ is strongly connected, then each vertex is a root.  

{\it d). Induced subgraphs}. 
Let $\Gamma = (V, E)$ be a digraph, and $V'$ be a subset of $V$. A subgraph $\Gamma' = (V', E')$ is said to be {\bf induced by $V'$} if the edge set $E'$ satisfies the following condition: let $v_i$ and $v_j$ be vertices in $V'$; then,  $v_i \to v_j$ is an edge of  $\Gamma'$ if and only if it is an edge of $\Gamma$. Note that if $\Gamma$ is a rooted graph with $V_r$ the set of roots, then the subgraph $\Gamma_r$ induced by $V_r$ is strongly connected.  

{\it e). Graph isomorphisms}.
Let $\Gamma = (V, E)$ and $\Gamma' = (V', E')$ be two digraphs. We say that $\Gamma$ is {\bf isomorphic} to $\Gamma'$ if there is a bijection $\sigma : V \longrightarrow V'$ such that the following condition holds: let $v_{i}$ and $v_j$ be any two vertices of $\Gamma$, then $v_i \to v_j$ is an edge of $\Gamma$ if and only if $\sigma(v_i)\to \sigma(v_j)$ is an edge of $\Gamma'$.  We call $\sigma$ a  {\bf graph isomorphism} between $\Gamma$ and $\Gamma'$.


\subsection{Voltage Graphs, Structural Balance and Local Groups}\label{ssec:preliminary}

In this subsection, we recall the definition of a voltage graph and a few other notions associated with it. We also describe some basic properties associated with a voltage graph.  
We start with the following definition:

\begin{Definition}[Voltage graphs]\label{def:voltagegraphs}
Let $G$ be a finite group. A {\bf voltage graph} is a pair $(\Gamma, \rho)$ with  $\Gamma = (V, E)$ a directed graph, and $\rho: E\longrightarrow G$ a map from the edge set $E$ to $G$. 
A voltage graph $(\Gamma, \rho)$ is {\bf weakly-connected}, {\bf rooted}, and {\bf strongly-connected}, respectively,  if $\Gamma$ is weakly-connected, rooted, and strongly-connected. Let $V'$ be a subset of $V$; a voltage graph $(\Gamma', \rho')$ is {\bf induced by} $V'$ if $\Gamma' = (V', E')$ is a subgraph of $\Gamma$ induced by $V'$ and $\rho': E' \to G$ is derived by restricting $\rho$ to the subset $E'$.  
\end{Definition}


To each voltage graph, one can associate a map which sends a semi-walk of $\Gamma$ to a group element, obtained as a multiplication of the group elements assigned to  the edges by the map $\rho$ along  the semi-walk.  Precisely, we have the following definition: 

\begin{Definition}[Net voltage~\cite{gross1987topological}]  
Let $(\Gamma, \rho)$ be a voltage graph, with $G$ the voltage group. Let ${SW}$ be the set of semi-walks of $\Gamma$. The {\bf net voltage} is  a map 
$
f: {SW} \longrightarrow G
$ 
defined as follows: 
Let $w$ be a semi-walk: 
\begin{equation*}\label{eq:seethewalkagain}
w = v_{i_1}\, a_{1}\, v_{i_2}\ldots v_{i_{n-1}}\, a_{n-1}\,  v_{i_n}.
\end{equation*} 
For each $j = 1,\ldots, n-1$, let 
$$
\ol \rho_w\(a_j\) :=
\left\{
\begin{array}{ll}
\rho(a_j)  & \mbox{if } a_{j} = e_{i_ji_{j+1}}, \\
\rho(a_j)^{-1} & \mbox{if } a_{j} = e_{i_{j+1}i_j}.
\end{array}
\right.  
$$ 
Then, set 
$$
f(w) := \ol \rho_w\(a_{1} \) \cdot  \ldots \cdot \ol \rho_w\( a_{n-1} \).
$$
For the case $w$ is trivial, set $f(w) := \mathbf{1}$.  We call $f(w)$  the {\bf net voltage on $w$}. 
\end{Definition}

\noindent

Note that the two operations on semi-walks---(i) concatenation and (ii) taking inverse---are both compatible with the voltage map. Precisely, 
we have the following fact:

\begin{lem}\label{lem:compatibility}
Let $(\Gamma, \rho)$ be a weakly connected voltage graph, and $w$ be a semi-walk of $\Gamma$. Then, the following hold:
\begin{enumerate} 
\item
Suppose that $w$ is a concatenation of $w'$ and $w''$, i.e., $w = w' \, w''$; then,  
$f(w) = f(w')\cdot f(w'')$.
\item For the inverse of $w$, we have $f(w^{-1}) = f(w)^{-1}.$
\end{enumerate}\,
\end{lem}

We omit the proof as the results directly follow from the definition of the net voltage.  
With the net voltage~$f$ at hand, we introduce the notion of structural balance:    

\begin{Definition}[Structural balance]\label{def:structuralbalance}
A voltage graph $(\Gamma, \rho)$ is {\bf structurally balanced} if $f(w) =  \mathbf{1}$ for any closed semi-walk $w$ in $\Gamma$. 
\end{Definition}

We note here that the notion of structural balance is originally defined for signed-graphs~\cite{cartwright1956structural}, and later extended to voltage graphs (see, for example,~\cite{rybnikov2005criteria}). 
We describe below a necessary and sufficient condition for a voltage graph to be structurally balanced. 
Recall that a semi-walk $w$ is said to be a  semi-cycle if there is no repetition of vertices of $w$, other than the repetition of the starting- and ending-vertex. We show below that a voltage graph is structurally balanced if and only if $f(w) = \mathbf{1}$ for any semi-cycle $w$ of $\Gamma$. Appealing to the operation of {\it cycle reduction} on closed semi-walks of $\Gamma$,  we establish the following rsult:

\begin{lem}\label{lem:conditionforstructuralbalance} 
Let $G$ be a finite group, and $(\Gamma, \rho)$ be a voltage graph. Then, $(\Gamma, \rho)$ is structurally balanced if and only if $f(c) = \mathbf{1}$ for all semi-cycles $c$ of $\Gamma$. 
\end{lem}

\begin{proof}
First, note that if $(\Gamma, \rho)$ is structurally balanced; then, from Definition~\ref{def:structuralbalance}, $f(w) = \mathbf{1}$ for all semi-cycles $w$ of $\Gamma$. We now show that the converse is also true. Let $w$ be closed semi-walk of $\Gamma$, and $w, w_1, \ldots, w_l$ be the chain of cycle reductions of $w$. Each $w_k$, for $k = 1,\ldots, l$, is obtained by removing a semi-cycle, denoted by $c_k$, from its predecessor. By assumption, we have $f(c_k) = \mathbf{1}$ for all $k = 1,\ldots, l$. It then follows that 
$$
f(w) = f(w_1) = \ldots = f(w_l) = \mathbf{1}.
$$   
The last equality holds because $w_l$ is itself a closed semi-walk of $\Gamma$. This completes the proof. 
\end{proof}


Let a voltage graph $(\Gamma, \rho)$  be {structurally unbalanced}. Then, from the definition, there is a closed semi-walk $w$ of $\Gamma$ such that $f(w) \neq \mathbf{1}$. If the voltage graph is a signed-graph, i.e., $G = \Z/(2)$,  then the value of $f(w)$ can only be $-1$. Yet, in the most general case where $G$ is an arbitrary finite group,   the value of $f(w)$ can be varied.     
We thus introduce for each vertex~$v_i$ a subgroup of $G$, termed  a {\it local group}, which characterizes all possible values of $f(w)$ for $w$ a closed semi-walk with $v_i$ the starting- and ending-vertex. Precisely, we have the following:

Let $\Gamma$ be a weakly connected digraph. Recall that $SW$ is the set of semi-walks of $\Gamma$. Let $v_i$ and $v_j$ be two vertices of $\Gamma$; we define $SW(v_i,v_j)$ to be the set of semi-walks of $\Gamma$ from~$v_i$ to~$v_j$. In particular, if $v_j = v_i$, then $SW(v_i,v_i)$ is the set of closed semi-walks of $\Gamma$ with $v_i$ the starting- and ending-vertex.



\begin{Definition}[Local groups]
Let $(\Gamma, \rho)$ be a weakly connected voltage graph, with $G$ the voltage group. 
For a vertex $v_i $ of $\Gamma$, let a subset  of $G$ be defined as follows: 
\begin{equation}\label{eq:defGi}
G_i := \left \{   f(w)   \mid  w\in SW(v_i,v_i)  \right\}.
\end{equation}
It is known that $G_i$ is a subgroup of $G$ (see, for example,~\cite{gross1987topological}). We call $G_i$ the {\bf local group} at $v_i$, and the collection $\{G_i\}_{v_i\in V}$ the  local groups of $(\Gamma, \rho)$.
\end{Definition}

It should be clear that the voltage graph $(\Gamma, \rho)$ is structurally balanced if and only if the local groups $G_i$, for  $v_i\in V$, are trivial subgroups of $G$.   
We further note  that any two local groups $G_i$ and $G_j$ are related by conjugation. Precisely, we have the following fact:

\begin{lem}[\cite{gross1987topological}]\label{lem:conjugation}
Let $(\Gamma,\rho)$ be a weakly connected voltage graph, with $G$ the voltage group. Let $v_i$ and $v_j$ be two vertices of $\Gamma$, and $w$ be a semi-walk from $v_i$ to $v_j$. Then, 
$$
G_j = f(w)^{-1} \cdot G_{i} \cdot f(w).
$$\,
\end{lem}

For the remainder of the subsection, we introduce the notion of a directed local group, which is a variation on the definition of a local group by restricting $f$ to closed walks of $\Gamma$. To proceed, let $W$ be the set of walks of $\Gamma$. Similarly, for two vertices $v_i$ and $v_j$, let $W(v_i,v_j)$ be the set of walks from~$v_i$ to~$v_j$. Then, we make the following definition:

\begin{Definition}[Directed local groups]
Let $(\Gamma, \rho)$ be a weakly connected voltage graph, with $G$ the voltage group. For a vertex $v_i$ of $\Gamma$, let $G^*_i$ be a subset of $G$ defined as follows:
$$
G^*_i = \{ f(w) \mid w\in W(v_i,v_i)\}. 
$$
We call $G^*_i$ the {\bf directed local group} at $v_i$, and the collection $\{G^*_i\}_{v_i \in V}$ the directed local groups of $(\Gamma, \rho)$. 
\end{Definition}
 
We show in the following lemma that each $G^*_i$ is indeed a subgroup of $G$. 

\begin{lem}\label{lem:Gistarisagroup}
Each $G^*_i$, for $v_i \in V$, is a subgroup of $G_i$.
\end{lem}

\begin{proof}
First, note that $G^*_i$ is a subset of $G_i$ because $W(v_i,v_i)$ is a subset of $SW(v_i,v_i)$. It thus suffices to show that $G^*_i$ is a subgroup of $G$. 
We need to show that (i) the identity element $\mathbf{1}$ is contained in $G^*_i$; (ii) for any two elements $g_1$ and $g_2$ in $G^*_i$, we have $g_1\cdot g_2\in G^*_i$; and (iii) for any $g\in G^*_i$, we have $g^{-1}\in G^*_i$.     
For~(i), note that the trivial walk $w = v_i$ is contained in $W(v_i,v_i)$, and hence $f(w) = \mathbf{1} \in G^*_i$. 
For~(ii), we first choose closed walks $w_1$ and $w_2$ in $W(v_i,v_i)$ such that 
$
f(w_i) = g_i$, for $i = 1, 2$.  
Let $w := w_1w_2$;  
then, $w\in W(v_i,v_i)$, and hence 
$$
f(w) = f(w_1)\cdot f(w_2) = g_1\cdot g_2 \in G^*_i.
$$
It now remains to establish~(iii). To proceed, note that 
since $G$ is a finite group, there exists a positive integer $m$, as the order of $g$, such that $g^m = \mathbf{1}$. In particular, 
$g^{m-1} = g^{-1} = \mathbf{1}$. Now, choose a $w\in W(v_i,v_i)$ such that $f(w) = g$, and let $w' :=w^{m-1}$. Then, $w'\in W(v_i,v_i)$, and moreover, 
$$f(w') = f(w)^{m-1} = g^{m-1} = g^{-1} \in G^*_i.$$
We have thus proved that $G^*_i$ is a subgroup of $G$. 
\end{proof}

We note here that $G^*_i$ is in general a proper subgroup of $G_i$. Also,  two directed local groups $G^*_i$ and $G^*_j$ may not be related by conjugation; indeed, the orders $|G^*_i|$ and $|G^*_j|$ may not be the same. We provide in Corollary~\ref{thm:propertiesforgstari} (in Subsection~\ref{ssec:onstrong}) sufficient conditions for (i) $G_i = G^*_i$, and (ii) $G^*_i$ and $G^*_j$ to be related by conjugation.


\subsection{On strongly connected voltage graphs}\label{ssec:onstrong}
In this subsection, we focus on the class of strongly connected voltage graphs, and establish certain relevant properties associated with it. 
To proceed, we first  define two  subsets of $G$. First, for any two vertices $v_i$ and $v_j$, let  
$\net(v_i,v_j)$ be defined as follows: 
$$
\net(v_i,v_j) = \{ f(w) \mid w\in SW(v_i,v_j)\}. 
$$
Note that if $v_j = v_i$, then $\net(v_i,v_i) $ is nothing but the local group $G_i$ at $v_i$.  
Recall that $W(v_i,v_j)$ is the set of walks from $v_i$ to $v_j$. 
Now, let a subset of $\net(v_i,v_j)$ be defined as follows: 
$$
\net^*(v_i,v_j) := \{f(w) \mid w\in W(v_i,v_j)\}. 
$$
In general, $\net^*(v_i,v_j)$ is only a proper subset of $\net(v_i,v_j)$. However, in the case when $(\Gamma, \rho)$ is strongly connected, we establish the following result: 

\begin{theorem}\label{pro:walkissufficient}
Let $(\Gamma, \rho)$ be a strongly connected voltage graph. Then, for any two vertices $v_i$ and $v_j$,
$$
\net(v_i,v_j) = \net^*(v_i,v_j). 
$$\,
\end{theorem}

\begin{proof}
Let $w$ be a semi-walk of $SW(v_i,v_j)$. 
It suffices to show that  there is a walk $w^* \in W(v_i,v_j)$ such that $f(w) = f(w^*)$.  
Suppose that $w$ is itself a walk; then we can let $w^* = w$.  We thus assume that $w$ is not a walk. Let  
$
w = v_{i_1} a_{1} \ldots a_{n-1} v_{i_n}
$, with $v_{i_1} = v_i$ and $v_{i_n} = v_j$. Then, there exists  an index $k = 1,\ldots, n - 1$ such that 
$
a_{k} = e_{i_{k+1} i_k}
$. Since $\Gamma$ is strongly connected, there is a path $p$ from $v_{i_k}$ to $v_{i_{k+1}}$. By concatenating the path $p$ with the edge $a_k$, we obtain a cycle $c = p\, a_k$ of $\Gamma$, with $v_{i_k}$ the starting- and ending-vertex. 
Since $G$ is a finite group, there exists a positive integer $m$, as the order of $f(c)$, such that 
$f(c)^m = \mathbf{1}$.

Let $w':= c^{m-1} p$; then, $w'$ is a walk from $v_{i_k}$ to $v_{i_{k+1}}$. 
Further, let a semi-walk $w_1$ from $v_i$ to $v_j$ be defined by replacing the edge $a_{k} $ in $w$ with the walk $w' $, i.e., 
$$
w_1 := v_{i_1} a_1 \ldots a_{k-1} v_{i_k} \xrightarrow{w'} v_{i_{k+1}} a_{k+1} \ldots a_{n-1} v_{i_n}. 
$$
Then, using the fact that $f(c)^m = f(w')\cdot \rho(a_k) = \1$, we obtain 
$$
\ol \rho_{w}(a_{k}) = \rho(a_k)^{-1} = f(w'), 
$$
and hence $f(w) = f(w_1)$. 
Recall that $l_-(w)$ is the total number of edges $a_k$ in $w$, with $a_{k} = e_{i_{k+1}i_k} $. From the construction of the semi-walk $w_1$, we have $l_-(w_1) =  l_-(w) - 1$.

Now, suppose that there exists another edge $a_{k'}$ in $w$ (and hence in $w_1$) such that $a_{k'} = e_{i_{k'+1}i_{k'}} $; then, by the same arguments above, we can obtain a new semi-walk $w_2\in SW(v_{i_1}, v_{i_n})$, by replacing the edge $a_{k'}$ in $w_1$ with a particularly chosen walk from $v_{i_{k'}}$ to $v_{i_{k'+1}}$, such that $f(w_2) = f(w_1) = f(w)$ and $l_-(w_2) = l_-(w_1) - 1 = l_-(w) - 2$.  
Continuing with this process, we then obtain,  in finite steps,  a walk $w^*$ from $v_{i}$ to $v_{j}$ with 
 $f(w^*) = f(w)$. This completes the proof. 
 \end{proof}






We state below some implications of Theorem~\ref{pro:walkissufficient}. Recall that from Lemma~\ref{lem:conditionforstructuralbalance}, a voltage graph is structurally balanced if and only if $f(c) = \1$ for each semi-cycle $c$ of $\Gamma$.  
Following Theorem~\ref{pro:walkissufficient},  we establish below a necessary and sufficient condition for a strongly connected voltage graph to be structurally balanced: 

\begin{cor}\label{cor:stronglyconnectedgraph}
Let $(\Gamma, \rho)$ be a strongly connected voltage graph. Then, $(\Gamma, \rho)$ is structurally balanced if and only if $f(c) = \mathbf{1}$ for each cycle $c$ of $\Gamma$.
\end{cor}

\begin{proof} First, note that a cycle $c$ of $\Gamma$ is a closed semi-walk, and hence if $(\Gamma, \rho)$ is a structurally balanced voltage graph, then $f(c) = \mathbf{1}$. We now show that the converse is also true.   
The proof is carried out by contradiction. Suppose that, to the contrary, there is a closed semi-walk  $w\in SW(v_i,v_i)$ such that $f(w) \neq \mathbf{1}$. 
Then, from Theorem~\ref{pro:walkissufficient},  there is a closed walk $w^* \in W(v_i,v_i)$ such that $$f(w^*) = f(w) \neq \mathbf{1}.$$ 
Let 
$
w^*_0, \, w^*_1, \, \ldots, \, w^*_l
$, with $w^*_0 = w^*$,  be the chain of cycle reductions of $w^*$, and let $c_i$, for $i = 1,\ldots, l$, be the cycle removed from $w^*_{i-1}$.  
Then, by the fact that $f(c_i) = \mathbf{1}$ for all $i = 1,\ldots, l$, we obtain 
$$
f(w^*) = f(w^*_1) = \ldots = f(w^*_l) \neq \1.
$$
On the other hand,  $w^*_l$ is itself a cycle of $\Gamma$, and hence 
$
f(w^*_l) = \mathbf{1}
$,   
which is a contradiction. This completes the proof. 
\end{proof}

Recall that for a vertex $v_i$ of a  digraph $\Gamma$, we have defined the directed local group at $v_i$ as 
$
G^*_i := \net^*(v_i,v_i) 
$. 
We have shown in Lemma~\ref{lem:Gistarisagroup} that $G^*_i$ is a subgroup of $G_i$. 
We now establish the following result as another corollary to Theorem~\ref{pro:walkissufficient}: 

\begin{cor}\label{thm:propertiesforgstari}
Let $(\Gamma, \rho)$ be a weakly connected voltage graph, with $G$ the voltage group.  Let $\{G^*_i\}_{v_i \in V}$ be the collection of directed local groups of $(\Gamma, \rho)$. 
Then, the following hold: 
\begin{enumerate}
\item If $\Gamma$ is strongly connected, then $G^*_i = G_i$. 
\item If two vertices $v_i$ and $v_j$ of $\Gamma$ are mutually reachable, then $G^*_i$ and $G^*_j$ are related by conjugation: let $w_{ij}$ be a walk from $v_i$ to $v_j$, then
\begin{equation}\label{eq:thm2cor2}
G^*_j = f(w_{ij})^{-1}\cdot G^*_i \cdot f(w_{ij}). 
\end{equation}  
\end{enumerate}\,
\end{cor}


\begin{proof}
The first part of the corollary directly follows from Theorem~\ref{pro:walkissufficient}; indeed, if $v_j = v_i$, then $$G_i =\net(v_i,v_i) = \net^*(v_i,v_i) = G^*_i.$$ 
We now prove the second part. Since $v_i$ and $v_j$ are mutually reachable, there is  a walk $w_{ji}$ from $v_j$ to $v_i$. Let $V'$ be the set of vertices incident to either $w_{ij}$ or $w_{ji}$, and let $\Gamma' = (V', E')$ be the subgraph induced by $V'$. Then, $\Gamma'$ is strongly connected.  From Theorem~\ref{pro:walkissufficient}, there is a walk $w'_{ji}$ of $\Gamma'$ from $v_j$ to $v_i$ such that $f(w'_{ji}) = f(w_{ij}^{-1}) = f(w_{ij})^{-1}$. Now, for each closed walk $w_{ii} \in W(v_i,v_i)$, we can define a closed walk $w_{jj} \in W(v_j,v_j)$ by
$
w_{jj} := w'_{ji} \,w_{ii}\, w_{ij}
$. 
This, in particular, implies that 
\begin{equation}\label{eq:askaliformeet}
G^*_j \supseteq f(w'_{ji}) \cdot G^*_i \cdot f(w_{ij}) = f(w_{ij})^{-1} \cdot G^*_i \cdot f(w_{ij}).
\end{equation}
Conversely, for any $w'_{jj}\in W(v_j,v_j) $, we can define a closed walk $w'_{ii} \in W(v_i,v_i)$ by
$
w'_{ii} := w_{ij} \,w'_{jj}\, w'_{ji} 
$, 
and hence 
\begin{equation}\label{eq:asltamerforreview}
G^*_i \supseteq f(w_{ij})\cdot G^*_j \cdot f(w'_{ji}) = f(w_{ij}) \cdot G^*_j\cdot f(w_{ij})^{-1}.
\end{equation}
Combining~\eqref{eq:askaliformeet} and~\eqref{eq:asltamerforreview}, we establish~\eqref{eq:thm2cor2}. %
\end{proof}

\subsection{On nondegenerate voltage graphs}


Let $(\Gamma, \rho)$ be a voltage graph, with $G$ the voltage group. If the map $\rho$ is such that  $\rho(e_{ij}) = \1$ for each edge $e_{ij}$ of $\Gamma$, then $(\Gamma,\rho)$ is said to be {\it trivial}. We are more interested in nontrivial voltage graphs. We introduce in this subsection a special class of nontrivial voltage graphs, termed  {\it nondegenerate} voltage graphs. Roughly speaking, a nondegenerate voltage graph exhibits all the elements  of the associated voltage group by going along semi-walks of the graph with a fixed starting vertex. We now define nondegenerate voltage graphs in precise term. First, recall that for any two vertices $v_i$ and $v_j$ of $\Gamma$, $SW(v_i,v_j)$ is the set of semi-walks from vertex $v_i$ to $v_j$, and $\net(v_i,v_j)$ is a subset of $G$ defined as follows:  
$$
\net(v_i,v_j) := \{f(w) \mid w\in SW(v_i, v_j)\}. 
$$
Now, we define
$$
SW(v_i, V) :=  \bigcup_{v_j \in V} SW(v_i,v_j).
$$
In other words, $SW(v_i,V)$ is the set of semi-walks of $\Gamma$ with $v_i$ the starting-vertex. Let
$$
\net(v_i,V) := \{f(w) \mid w \in SW(v_i,V)\} =  \bigcup_{v_j \in V} \net(v_i,v_j).
$$ 
We then have the following definition:

\begin{Definition}[Nondegenerate voltage graphs]\label{def:nondegenerate}
Let  $(\Gamma, \rho)$ be a weakly connected voltage graph, with $G$ the voltage group. Then, $(\Gamma, \rho)$ is {\bf nondegenerate} 
 if there exists a vertex $v_i\in V$ such that $\net(v_i,V) = G$. 
\end{Definition}

\noindent 
Note that the definition above does not depend on a particular choice of a vertex $v_i$; indeed, we have the following fact: 

\begin{lem}\label{lem:nondegGgraph}
Let $(\Gamma, \rho)$ be a  weakly connected voltage graph. Then, for any two vertices $v_i$ and $v_j $ of $\Gamma$, we have $|\net(v_i, V)| =|\net(v_j, V)|$.  
\end{lem}

\begin{proof}
We first show that $|\net(v_j, V)| \ge |\net(v_i, V)|$, and then show that $|\net(v_j, V)| \le |\net(v_i, V)|$. Since $\Gamma$ is weakly connected,  there is a semi-walk $w$  from $v_j$ to $v_i$. Hence, for any semi-walk $w_i$ in $SW(v_i, V)$, we derive a semi-walk $w_j$ in $SW(v_j, V)$ by concatenating $w$ and $w_i$, i.e., $w_j := w\, w_i$. This, in particular, implies that 
$$
\net(v_j,V) \supseteq f(w) \cdot \net(v_i,V),
$$   
and hence $|\net(v_j,V) | \ge |\net(v_i,V)|$. Applying the same arguments, we obtain 
$$\net(v_i,V) \supseteq f(w^{-1}) \cdot \net(v_j,V),$$ 
and hence $|\net(v_i,V)| \ge |\net(v_j,V)|$.   
This completes the proof.  
\end{proof}

We describe below a necessary and sufficient condition for a voltage graph to be nondegenerate. First, recall that given a group element $g$ of $G$ and  a subgroup $H$, the right-coset of $H$ with respect to $g$ is given by
$H\cdot g$. We establish the following result:

\begin{pro}\label{pro:necessaryandsufficientcondfornond}
Let $(\Gamma, \rho)$ be a voltage graph, with $G$ the voltage group. Then, the following properties hold:
\begin{enumerate}
\item Let $v_i$ and $v_j$ be two vertices of $\Gamma$, and $w$ be a semi-walk from $v_i$ to $v_j$. Then, 
\begin{equation}\label{eq:xiedaoxianzai}
\net(v_i,v_j) = G_i \cdot f(w). 
\end{equation}
In particular, $|G_i|$ divides $|\net(v_i,V)|$.
\item Fix a vertex~$v_i$; then, $(\Gamma, \rho)$ is nondegenerate if and only if for any $g \in G$, there exists a vertex $v_j$ such that $\net(v_i,v_j) = G_i \cdot g$.   
\end{enumerate}\,
\end{pro}

\begin{proof}
We prove part~1 of the proposition by establishing the following two inequalities: $\net(v_i,v_j) \supseteq G_i \cdot f(w)$ and  $\net(v_i,v_j) \subseteq G_i \cdot f(w)$. To establish the first inequality, let $g\in G_i \cdot f(w)$; then, $g = f(w') \cdot f(w)$ 
for $w'$ a closed semi-walk in $SW(v_i.v_i)$. Since the concatenation $w'w$ is a semi-walk from $v_i$ to $v_j$, we have 
$$g = f(w') \cdot f(w) = f(w'w) \in \net(v_i,v_j).$$
To establish the second inequality, let $g\in \net(v_i,v_j)$, and $w_g$ be a semi-walk from~$v_i$ to~$v_j$ such that $f(w_g) = g$. Let $w' = w_g w^{-1}$; then, $w'$ is a closed semi-walk in $SW(v_i,v_i)$, and hence 
$$
g = f(w_g w^{-1}) \cdot f(w) \in G_i \cdot f(w).   
$$
We have thus established~\eqref{eq:xiedaoxianzai}. Then, using the fact that two right-cosets $G_i \cdot g$ and $G_i \cdot g'$ are either identical or disjoint, we have that $|G_i|$ divides  $|\net(v_i,V)|$.

The second part then directly follows from the first part; indeed, from~\eqref{eq:xiedaoxianzai}, we obtain the following relation: 
\begin{equation}\label{eq:notveryhappybutok}
\{\net(v_i,v_j) \mid v_j \in V\} \subseteq G_i \backslash G,
\end{equation}
where we recall that $G_i \backslash G$ is the collection of right-cosets of $G_i$. Since the right-cosets of $G_i$ form a partition of $G$, we conclude that $(\Gamma, \rho)$ is nondegenerate if and only if the equality holds in~\eqref{eq:notveryhappybutok}. 
\end{proof}

In the remainder of the subsection, we focus on voltage graphs that are both structurally balanced and nondegenerate. In particular, we investigate the following question: given a weakly connected digraph $\Gamma$ and a finite group $G$, does there exist a map $\rho: E\longrightarrow  G$ such that $(\Gamma, \rho)$ is both structurally balanced and nondegenerate? We provide a complete answer to this question in the following theorem: 


\begin{theorem}\label{thm:nondegenerateplusstructuralbalance}
Let $\Gamma = (V, E)$ be a weakly connected voltage graph, with $G$ the voltage group. Then, there exists a map $\rho: E\longrightarrow G$ such that the voltage graph $(\Gamma, \rho)$ is  structurally balanced and nondegenerate if and only if $|V| \ge  |G| $. 
\end{theorem}

\begin{proof}
We first show that if $|V| < |G|$, then there does not exist a map $\rho: E\longrightarrow G$ such that $(\Gamma, \rho)$ is structurally balanced and nondegenerate. Suppose that, to the contrary, these exists such a map $\rho$; then, for any vertex $v_i$, there exists a vertex $v_j$, together with two semi-walks $w_1$ and $w_2$ from $v_i$ to $v_j$, such that $f(w_1) \neq f(w_2)$. This holds because otherwise, $|\net(v_i,V)| \le |V| < |G|$, and hence $(\Gamma, \rho)$ is degenerate. 
By concatenating $w_1$ with $w^{-1}_2$, we obtain a closed semi-walk $w := w_{1}\, w^{-1}_{2}$,  with $v_i$ the starting- and ending-vertex, and moreover, 
$$
f(w) = f(w_1)\cdot f(w^{-1}_2) \neq \1.
$$
Hence, $(\Gamma, \rho)$ is structurally unbalanced, which is a contradiction. 

We now show that if $|V| \ge |G|$, then there is a map $\rho: E\longrightarrow G$ such that the voltage graph $(\Gamma, \rho)$ is structurally balanced and nondegenerate. Since $|V| \ge |G|$, there is a {\it surjective} map $\eta: V\longrightarrow G$. 
Fix any such map $\eta$, and let $\rho: E \longrightarrow G$ be defined as follows: for an edge $e_{ij}$ of $\Gamma$, let
\begin{equation}\label{eq:resultingmaprho}
\rho(e_{ij}):= \eta(v_i)^{-1} \cdot \eta(v_j). 
\end{equation}
We now show that the voltage graph $(\Gamma, \rho)$, with the map $\rho$ defined above, is structurally balanced and  nondegenerate. First, note that by the construction of the map $\rho$, the net voltage satisfies the following condition: let $w$ be a semi-walk from $v_i$ to $v_j$, then 
\begin{equation}\label{eq:relation01}
f(w) = \eta(v_i)^{-1} \cdot \eta(v_j).  
\end{equation}
Hence, for any closed semi-walk $w$, we have $f(w) = \1$, which implies that $(\Gamma, \rho)$ is structurally balanced. We next show that $(\Gamma, \rho)$ is nondegenerate. Since $\eta: V \longrightarrow G$ is surjective, there is a vertex $v_i$ such that $\eta(v_i) = \1$.  Then, from~\eqref{eq:relation01}, we conclude that if $w$ is a semi-walk from $v_i$ to $v_j$, then 
$
f(w) = \eta(v_j)
$. In other words, 
$$
\net(v_i,V) = \{\eta(v_j) \mid v_j \in V\} = G,
$$
which implies that $(\Gamma, \rho)$ is nondegenerate.
\end{proof}

The proof of Theorem~\ref{thm:nondegenerateplusstructuralbalance} further implies the following: first, recall that for a pair of positive integers $(n,k)$, with $n \ge k$, a {\bf Stirling number of the second kind}, denoted by $S(n,k)$, is given by
$$
S(n,k) := \frac{1}{k !}\sum^k_{i=0} (-1)^{k - i} {k \choose i} i^n,   
$$   
which can be viewed as the number of ways to partition a set of $n$ objects into $k$ non-empty subsets; with the number $S(n,k)$ at hand, we establish the following result as a corollary to Theorem~\ref{thm:nondegenerateplusstructuralbalance}: 

\begin{cor}
Let $\Gamma = (V, E)$ be a weakly connected, and $G$ be a finite group. Suppose that $|V| \ge |G|$; then, there are as many as $ S(|V|, |G|) (|G|-1)!$  
different maps $\rho: E\longrightarrow G$ such that the voltage graph $(\Gamma, \rho)$ is nondegenerate. 
\end{cor}

\begin{proof}
First, note that from the proof of Theorem~\ref{thm:nondegenerateplusstructuralbalance}, if the map $\eta: V\longrightarrow G$ is surjective, then the resulting map $\rho$, defined by~\eqref{eq:resultingmaprho}, yields a structurally balanced and nondegenerate voltage graph $(\Gamma, \rho)$. Conversely, each structurally balanced and nondegenerate voltage graph $(\Gamma, \rho)$ can be constructed in this way. To see this, we first fix a vertex~$v_1$ of $\Gamma$, and let $\eta(v_1) = \1$; then, for any vertex~$v_i$ of $\Gamma$, we choose a semi-walk~$w$ from~$v_1$ to $v_i$, and set $\eta(v_i):= f(w)$. Note that the definition of $\eta(v_i)$ does not depend on a particular choice of the semi-walk~$w$ because $(\Gamma, \rho)$ is structurally balanced. 

The arguments above then imply the following fact: let $N_1$ be the number of maps $\rho: E\longrightarrow G$ such that $(\Gamma, \rho)$ is structurally balanced and nondegenerate, and $N_2$ be the number of maps $\eta: V \longrightarrow  G$ that are surjective; then, $N_1 = N_2/|G|$. To see this, note that $N_1$ is exactly the number of surjective maps $\eta: V\longrightarrow G$ with $\eta(v_1) = \1$, which is then given by $N_2/|G|$.     

It now suffices to compute~$N_2$.  To proceed, note that a surjective map~$\eta$ can be constructed in two steps: first, we partition the vertex set $V$ into $|G|$ non-empty subsets $V_1, \ldots, V_{|G|}$; then, we assign a group element $g_i$ to the vertices of $V_i$, and the assignment is such that $g_1,\ldots, g_{|G|}$ are pairwise distinct. It then follows that $N_2 = S(|V|, |G|) |G|! $, and hence $N_1 = S(|V|, |G|)  (|G| - 1)!$. 
\end{proof}

\section{Derived Graphs and Root-Connectivity of Their Connected Components}\label{ssec:rootconnectivity}

In this section, we introduce an  important object associated with a voltage graph, namely the {\it derived graph}. To proceed, we first recall the notion of a covering graph. Let $\Gamma = (V, E)$ and $\ol \Gamma = (\ol V, \ol E)$ be two arbitrary digraphs, and let $\pi : \ol V \longrightarrow V$ be a surjective map. Then, we say that $\ol \Gamma$ is a {\bf covering graph} of $\Gamma$ (correspondingly, $\pi$ is a {\bf covering map}) if for each vertex $\ol v \in \pi^{-1}(v)$, the numbers of in- and out-neighbors of $\ol v$ in $\ol \Gamma$ are the same as those of $v$ in $\Gamma$.  
In other words, the local structure of $\ol \Gamma$ at $\ol v$ is identical with the local structure of $\Gamma$ at $v$.  The derived graph of $(\Gamma,\rho)$ is then a particular covering graph of $\Gamma$. Precisely, we have the following definition: 

\begin{Definition}[Derived graph]\label{def:coveringgraph}
 Let $(\Gamma, \rho)$ be a voltage graph, with $G$ the voltage group. The {\bf derived graph} $\ol \Gamma = (\ol V, \ol E)$ of $(\Gamma, \rho)$ is a covering graph of $\Gamma$ with $|G| |V|$ vertices and $|G| |E| $ edges. Specifically, we have the following:  
 \begin{enumerate}
\item  The vertex set of $\ol \Gamma$ is 
 $
 \ol V= \{ [g, v_i] \mid g\in G, v_i \in V\}
 $. 
\item  The edge set of $\ol \Gamma$ is determined by the following condition: 
 $
 [g_i, v_i] \to [g_j, v_j] 
 $
 is an edge of $\ol \Gamma$  if and only if $e_{ij}$  is an edge of $\Gamma$ and   
 $g_j = g_i \cdot  \rho(e_{ij})$. 
 \end{enumerate}\,
\end{Definition}

Note that a derived graph $\ol \Gamma$ is indeed a covering graph of $\Gamma$. To see this, let the projection map $\pi: \ol V\longrightarrow V$ be defined  as follows: 
\begin{equation}\label{eq:defprojectionmappi}
\pi: [g, v_i] \mapsto v_i. 
\end{equation}
Then, for each vertex $v_i\in V$, the pre-image $\pi^{-1}(v_i)$ is given by
$
\pi^{-1}(v_i)= G\times \{v_i\} 
$. 
Moreover, the in- and out-neighbors of each vertex $[g, v_i] \in \pi^{-1}(v_i)$ are given by
$$
\left\{
\begin{array}{l}
\cal{N}^-([g,v_i]) = \left \{ [g\cdot \rho(e_{ij}) , v_j]  \mid e_{ij} \in E\right \} \\
\cal{N}^+([g,v_i]) = \left \{ [g\cdot \rho(e_{ki})^{-1}, v_k]  \mid e_{ki} \in E\right \},  
\end{array}
\right.
$$
and hence the numbers of in- and out-neighbors of $[g, v_i]$ are the same as those of $v_i$ in~$\Gamma$. 

\subsection{On connected components of a derived graph}
Let $(\Gamma, \rho)$ be a voltage graph, with $G$ the voltage group,  and $\ol \Gamma$ be the associated derived graph. In general, $\ol \Gamma$ is not connected; indeed, we will see soon that $\ol \Gamma$ is (weakly) connected if and only if the local groups $\{G_i \}_{v_i\in V}$ of $(\Gamma, \rho)$ are such that $G_i = G$ for all $v_i \in V$. 
Suppose that $\ol \Gamma$ is not connected; then, it must be comprised of multiple weakly connected subgraphs.  
We call each connected subgraph  a {\bf connected component} of $\ol \Gamma$. 
In this subsection, we describe certain relevant properties associated with the connected components of a derived graph. To proceed, we first recall that for a group element $g\in G$,  the left-coset of $G_i$ with respect to~$g$ is given by $g \cdot G_1$. We start with the following fact: 


\begin{lem}\label{lem:couldbeuseful}
Let $[g, v_i]$ and $[g',v_j]$ be two vertices of the derived graph $\ol \Gamma$. Then, the following hold: 
\begin{enumerate}
\item There is a semi-walk  (resp. walk) from $[g, v_i]$ to $[g',v_j]$ if and only if there is a semi-walk (resp. walk) $w$ of $\Gamma$ from~$v_i$ to~$v_j$ such that $g' = g \cdot f(w)$. \item If $v_j = v_i$, then $[g,v_i]$ and $[g', v_i]$ belong to the same connected component if and only if $g$ and $g'$ belong to the same left-coset of~$G_i$.   
\end{enumerate}
\end{lem} 

\begin{proof}
The first part of the lemma directly follows from Definition~\ref{def:coveringgraph}. For the second part, first note that $[g, v_i]$ and $[g',v_i]$ belong to the same component if and only if $g = g' \cdot f(w)$ for~$w$ a closed semi-walk in $SW(v_i,v_i)$. Since $f(w)\in G_i$, we thus conclude that $g$ and $g'$ belong to the same left-coset of~$G_i$.  
\end{proof}





For the remainder of the subsection, we fix a vertex $v_1$ of $\Gamma$, and let $G_1$ be the local group at $v_1$. 
Let $k$ be the index of $G_1$ in $G$, and let $g_1,\ldots, g_k\in G$ be chosen such that the left-cosets  $g_1 \cdot G_1,\ldots, g_k \cdot G_1$ partition the group $G$.  From the second part of Lemma~\ref{lem:couldbeuseful}, $[g_i,v_1]$ and $[g_j, v_1]$, for $i\neq j$,  belong to two different connected components of $\ol \Gamma$. In other words, there are at least~$k$ connected components of $\ol \Gamma$ (we will see soon that  the number~$k$ is actually exact). To proceed,  recall that two digraphs  $\Gamma = (V, E)$ and $\Gamma' = (V', E')$ are said to be isomorphic if there is a bijection $\sigma : V \longrightarrow V'$, termed a graph isomorphism, such that for any two vertices $v_{i}$ and $v_j$ of $\Gamma$, $v_i \to v_j$ is an edge of $\Gamma$ if and only if $\sigma(v_i)\to \sigma(v_j)$ is an edge of $\Gamma'$.   
We now establish the following result:

\begin{pro}\label{lem:connectedcomponentsofderivedgraph}
Let $(\Gamma,\rho)$ be a weakly connected voltage graph, with $G$ the voltage group. Let $G_1$ be the local group at vertex~$v_1$, and $k:= |G|/|G_1|$ the index of $G_1$ in $G$. Let $g_1,\ldots, g_k$ be chosen such that $G = \bigsqcup^k_{i = 1} (g_i \cdot G_1)$. 
Then, the following hold for the connected components of the associated derived graph~$\ol \Gamma$:
\begin{enumerate}
\item There are $k$ connected components of $\ol \Gamma$, labelled as $\ol \Gamma_1 = (\ol V_1, \ol E_1),\ldots,\ol \Gamma_k = (\ol V_k, \ol E_k)$.   
Any two connected components are isomorphic: without loss of generality, let $[g_i, v_1]\in \ol V_i$, for all $i = 1,\ldots, k$; then, the map $$\sigma_{ij}: [g,v] \mapsto [g_j \cdot g^{-1}_i\cdot g, v], $$ 
when restricted to $\ol V_i$ is a graph isomorphism between $\ol \Gamma_i$ and $\ol \Gamma_j$.    
\item If $(\Gamma, \rho)$ is structurally balanced, then each connected component  is isomorphic to $\Gamma$. The projection map $$\pi_i: [g, v]  \mapsto  v,$$ 
when restricted to $\ol V_i$  is a graph isomorphism between $\ol \Gamma_i$ and $\Gamma$.  
\end{enumerate}\,   
\end{pro}

\begin{proof}
We first prove part~1 of the proposition. From the second part of Lemma~\ref{lem:couldbeuseful}, there exist at least~$k$ connected components of $\ol \Gamma$. To show that~$k$ is exact, it suffices to show that for any vertex $[g, v]$ of $\ol \Gamma$, there is a semi-walk from $[g,v]$ to $[g',v_1]$ for some $g'\in G$.  Let~$w$ be a semi-walk of $\Gamma$ from~$v_i$ to~$v$; then, from the first part of Lemma~\ref{lem:couldbeuseful}, there is a semi-walk from $[g, v]$ to $[g \cdot f(w), v_1]$. 

We next show that $\ol \Gamma_i = (\ol V_i, \ol E_i)$ and $\ol \Gamma_j = (\ol V_j, \ol E_j)$ are isomorphic, with $\sigma_{ij}: \ol V_i\longrightarrow \ol V_j$ a graph isomorphism. 
First, note that $\sigma_{ij}$ is indeed a bijection between $\ol V_i$ and $\ol V_j$.  To see this, let $[g, v]\in \ol V_i$, and we show that $[g_j \cdot g^{-1}_i \cdot g, v] \in \ol V_j$.   
Since $[g_i, v_1]\in \ol V_i$, from the first part of Lemma~\ref{lem:couldbeuseful}, there is a semi-walk~$w$ from~$v_1$ to~$v$ such that $f(w) = g^{-1}_i \cdot g$, which in turn implies that there is a semi-walk from $[g_j, v_1]$ to 
$$[g_j \cdot f(w), v] = [g_j \cdot g^{-1}_i \cdot g, v] \in \ol V_j. $$  
It then follows that $\sigma_{ij}$ is a graph isomorphism: let $[g,v_a] \to [g\cdot \rho(e_{ab}), v_b]$ be an edge of $\ol \Gamma_i$; then, $$[g_j \cdot g^{-1}_i\cdot g, v_a] \to [g_j \cdot g^{-1}_i\cdot g\cdot \rho(e_{ab}), v_b]$$
is an edge of $\ol \Gamma_j$, and vice versa. We have thus established the first part of the proposition. 

To establish the second part, first note that in the case $(\Gamma, \rho)$ is structurally balanced, we have $|G_1| = 1$, and hence there are $|G|$ connected components of $\ol \Gamma$, each of which has $|V|$ vertices and $|E|$ edges. In particular, the projection map $\pi_i$ is a bijection between $\ol V_i$ and $V$. Now, let $[g, v_a]$ be a vertex of $\ol \Gamma_i$, and  $[g, v_a] \to [g', v_b]$ be an edge of $\ol \Gamma_i$; then, it should be clear that $e_{ab}$ is an edge of $\Gamma$, and moreover, $g' = g \cdot \rho(e_{ab})$.  Conversely, if $e_{ab}$ is an edge of $\Gamma$, then by the fact that $\pi_i$ is a bijection, we conclude that $\pi^{-1}(v_a) = [g, v_a]$ and $\pi^{-1}(v_b) = [g\cdot \rho(e_{ab}), v_b]$, and hence
$$
\pi^{-1}(v_a) \to \pi^{-1}(v_b) = [g, v_a] \to [g\cdot \rho(e_{ab}), v_b]
$$      
is an edge of $\ol \Gamma_i$. This completes the proof. 
\end{proof}

\subsection{On root-connectivity of the connected components}
In this subsection, we assume that a voltage graph is rooted, and investigate the root connectivity of each connected component of the associated derived graph.  To proceed, we first recall some proven results about the collection of directed local groups  $\{G^*_i\}_{v_i\in V}$: (i) we have shown that   
each $G^*_i$ is a subgroup of the local group $G_i$; (ii) we have also shown that if $v_i$ and $v_j$ are mutually reachable,  then $G^*_i$ and $G^*_j$ are related by conjugation. Now, let $\Gamma$ be a rooted graph, and  $v_i$ and $v_j$ be two roots  of $\Gamma$. Then, $v_i$ and $v_j$ are mutually reachable, and hence $G^*_i$ and $G^*_j$ are related by conjugation. This, in particular, implies that if $G^*_i = G_i$ for some root $v_i$ of $\Gamma$, then $G^*_j = G_j$ for all roots  $v_j$.   
With the preliminaries above, we establish the following result: 






\begin{theorem}\label{thm:coveringgraph}
Let $(\Gamma, \rho)$  be a rooted voltage graph, with $G$ the voltage group.  Let $\{G_i\}_{v_i\in V}$ (resp. $\{G^*_i\}_{v_i\in V}$) be the  local groups (resp. directed local groups) of $(\Gamma,\rho)$. Let $\ol \Gamma$ be the derived graph of $(\Gamma, \rho)$. Then, the connected components of $\ol \Gamma$ are rooted if and only if $G^*_i = G_i$ for some (and hence any)  root $v_i$ of $\Gamma$.   
\end{theorem}


The remainder of the subsection is devoted to the proof of Theorem~\ref{thm:coveringgraph}. We first prove for the case where $\Gamma$ is strongly connected: 

\begin{lem}\label{lem:stronglyconnected}
Let $(\Gamma, \rho)$ be a strongly connected voltage graph, and $\ol \Gamma$ be the associated derived graph. Then,  each connected component $\ol \Gamma_i$, for $i = 1,\ldots, k$, of $\ol \Gamma$ is strongly connected. 
\end{lem}

\begin{proof}
Let $[g', v_a]$ and $[g'', v_b]$ be two vertices of $\ol \Gamma_i$. It suffices to show that there is a walk of $\ol \Gamma_i$ from $[g', v_a]$ to $[g'',v_b]$. 
From the first part of Lemma~\ref{lem:couldbeuseful}, there is a semi-walk~$w$ from~$v_a$ to~$v_b$ such that $g'' = g' \cdot f(w) $. Since $\Gamma$ is strongly connected, from Theorem~\ref{pro:walkissufficient}, there is a walk~$w'$ from~$v_a$ to~$v_b$ such that
$f(w') = f(w)$, and hence $g'' = g' \cdot f(w')$. Appealing again to the first part of Lemma~\ref{lem:couldbeuseful}, we conclude that there is walk of $\ol \Gamma_i$ from $[g', v_a]$ to $[g'',v_b]$. 
\end{proof}


We now focus on the case where the voltage graph $(\Gamma, \rho)$ is only  rooted.  Denote by $V_r$ the set of roots of $\Gamma$. 
Let $(\Gamma_r, \rho_r)$ be the voltage graph induced by $V_r$---the digraph $\Gamma_r$ is a subgraph of $\Gamma$ induced by $V_r$ and the map $\rho_r$ is derived by restricting $\rho$ to $V_r$. 
Let $SW_r$ (resp. $W_r$) be the set of semi-walks (resp. walks) of $\Gamma_r$. Similarly, for vertices $v_i$ and $v_j$ of $\Gamma_r$, let $SW_r(v_i,v_j)$ (resp. $W_r(v_i,v_j)$) be the set of semi-walks (resp. walks) from $v_i$ to $v_j$.  
We state below some facts about the voltage graph $(\Gamma_r, \rho_r)$. 
First, let a subset of $\ol V$ be defined as follows:
$$
\ol V_r := \pi^{-1}(V_r) = \{ [g, v] \mid g\in G, v\in V_r\}. 
$$
 Let $\ol \Gamma_r = (\ol V_r, \ol E_r)$ be the subgraph of $\ol \Gamma$ induced by $\ol V_r$.  It then directly follows from Definition~\ref{def:coveringgraph} that $\ol \Gamma_r$ is  the derived graph of $(\Gamma_r, \rho_r)$.  
We further establish the following result:

\begin{lem}
The  local groups of $(\Gamma_r, \rho_r)$ are $\{G^*_i\}_{v_i\in V_r}$.
\end{lem}

\begin{proof}
Let $v_i$ be a root of $\Gamma$; we first show that 
\begin{equation}\label{eq:wequalwr}
W(v_i,v_i) = W_{r}(v_i,v_i). 
\end{equation}
It suffices to show that each closed-walk $w$ in $W(v_i,v_i)$ is indeed in $W_r(v_i,v_i)$. This holds because (i) the starting vertex of $w$ is $v_i$, which is a root,  and (ii) an out-neighbor of a root is also a root. Thus, all the vertices in $w$ are roots of $\Gamma$, which implies that $w\in W_r(v_i,v_i)$. 
Following~\eqref{eq:wequalwr}, we obtain
\begin{equation}\label{eq:10:16pm}
G^*_i = \{f(w) \mid w\in W_r(v_i,v_i)\}.
\end{equation}
On the other hand, $(\Gamma_r, \rho_r)$ is strongly connected. We thus appeal to Theorem~\ref{pro:walkissufficient}, and obtain
\begin{equation}\label{eq:youarelate}
\{f(w) \mid w\in W_r(v_i,v_i)\} = \{f(w) \mid w\in SW_r(v_i,v_i)\}.   
\end{equation}
Combining~\eqref{eq:10:16pm} and~\eqref{eq:youarelate}, we conclude that for all $v_i\in V_r$, 
$$
G^*_i  = \{f(w) \mid w\in SW_r(v_i,v_i)\},
$$
and hence  $\{G^*_i\}_{v_i \in V_r}$ are the local groups of $(\Gamma_r, \rho_r)$.  
\end{proof}

Let $v_i$ be a root of $\Gamma$.  Since $G^*_i$ is a subgroup of $G$, we have that $|G^*_i|$ divides $|G|$. Now, let  
$$
k^* := |G| / |G^*_i|.    
$$  
Since two roots $v_i$ and $v_j$ of $\Gamma_r$ are mutually reachable, from Corollary~\ref{thm:propertiesforgstari}, $G^*_i$ and $G^*_j$ are related by conjugation. In particular, $|G^*_i| = |G^*_j|$, and hence the number $k^*$ does not depend on a particular choice of a root $v_i$ of $\Gamma$. We also note that $G^*_i$ is a subgraph of $G_i$; since $k = |G| / |G_i|$, we have that $k$ divides $k^*$, and moreover, $k = k^*$ if and only if $G^*_i = G_i$. For convenience, let $$l := k^* / k.$$ 
From Proposition~\ref{lem:connectedcomponentsofderivedgraph}, there are $k^*$  connected components of $\ol \Gamma_r$, labelled as $\ol \Gamma_{r,1},\ldots, \ol \Gamma_{r, k^*}$, any two of which are isomorphic.  Furthermore, since $(\Gamma_r, \rho_r)$ is strongly connected, from Lemma~\ref{lem:stronglyconnected}, each $\ol \Gamma_{r, i}$, for $i = 1,\ldots, k^*$, is strongly connected.  
On the other hand, $\ol \Gamma_{r}$ is a subgraph of $\ol \Gamma$ induced by $\ol V_r$. Since $\ol \Gamma$ has only $k$ isomorphic connected components, we have the following fact:  

\begin{lem}\label{lem:thelemmaofpartition}
There exists a partition of the index set $\{1,\ldots, k^*\}$:
\begin{equation*}\label{eq:partitionsomehow}
\{1,\ldots, k^*\} = \sqcup^k_{j = 1}\cal{I}_{j},
\end{equation*} 
such that the following hold:
\begin{enumerate}
\item $|\cal{I}_j| = l$ for all $j = 1,\ldots, k$.   
\item Each $\ol \Gamma_{r, i}$,  for $i\in \cal{I}_j$,  is  a subgraph of $\ol \Gamma_j$. 
\end{enumerate}
\end{lem}

\begin{proof}
Without loss of generality, we assume that $v_1$ is a root of $\Gamma$. Let $G_1$ (resp. $G^*_1)$ be the local group (resp. directed local group) at~$v_1$. 
Let $g^*_1, \ldots, g^*_{k^*}$ be chosen such that $G = \bigsqcup^{k^*}_{i = 1} (g^*_i\cdot G^*_1)$; we assume, without loss of generality,  that $[g^*_i,  v_1] $ is a vertex of $\ol \Gamma_{r,_i}$, for all $i = 1,\ldots, k^*$. It now suffices to show that for each subgraph $\ol\Gamma_j$, there exists a subset $\cal{I}_j\subset \{1,\ldots, k^*\}$, with $|\cal{I}_j| = l$,  such that the vertices $[g^*_i, v_1]$, for $i \in \cal{I}_j$, belong to the subgraph $\ol \Gamma_j$. Let $g_1,\ldots,g_k$ 
be chosen as in Proposition~\ref{lem:connectedcomponentsofderivedgraph}: we have $G = \bigsqcup^k_{i = 1}(g_i \cdot G_1)$ and each $[g_i, v_1]$, for $i = 1,\ldots, k$, is a vertex of $\ol \Gamma_i$.  Because $G^*_1$ is a subgroup of $G_1$, with~$l$ the index of $G^*_1$ in $G_1$. So, each $g^*_i\cdot G^*_1$, for $i = 1,\ldots, k^*$, is a subset of $g_j \cdot G_1$ for some $j = 1,\ldots, k$. This, in particular, implies that there exists a subset $\cal{I}_j$, with $|\cal{I}_j| = l$, such that $g^*_i \cdot G^*_1 \subset g_j \cdot G_1$ for all $i \in \cal{I}_j$. Since $$g^*_i \in g^*_i \cdot G^*_1 \subset g_j \cdot G_1,$$ we conclude, from the second part of Lemma~\ref{lem:couldbeuseful}, that all vertices $[g^*_i, v_1]$, for $i \in \cal{I}_j$, belong to the same subgraph $\ol \Gamma_j$, which completes the proof. 
\end{proof}

Following Lemma~\ref{lem:thelemmaofpartition}, we consider below two cases about the value of the integer~$l$:

\vspace{3pt}
\noindent
{\bf Case I}. We assume that $l > 1$, and prove that $\ol \Gamma_{j}$ is not rooted. Specifically, we prove the following fact: 

\begin{lem}\label{lem:feeluncomfortable}
Let $\ol \Gamma_{r,i}$ be a subgraph of $\ol \Gamma_j$. Let $[g, v]$ be a vertex of $\ol \Gamma_{r, i}$, and $[g',v']$ be a vertex of $\ol \Gamma_j$. Then, there is a walk from $[g,v]$ to $[g',v']$ if and only if $[g',v']$ is a vertex of $\ol \Gamma_{r,i}$. 
\end{lem}

\begin{proof}
Since $\ol \Gamma_{r, i}$ is strongly connected,  if $[g',v']$ is a vertex of $\ol \Gamma_{r,i}$, then there is a walk from $[g, v]$ to $[g',v']$. We now show that if $[g', v']$ is not a vertex of $\ol \Gamma_{r, i}$, then there does not exist a walk from $[g, v]$ to $[g',v']$.  The proof is carried out by contradiction: we assume, to the contrary, that such walk from $[g,v]$ to $[g', v']$ exists. Then, from the first part of Lemma~\ref{lem:couldbeuseful}, there is a walk $w$ of $\Gamma$ from $v$ to $v'$, with $g' = g\cdot f(w)$. 
 Since $[g, v]$ is a vertex of $\ol \Gamma_{r,i}$, and hence of $\ol \Gamma_r$,  $v$ is a root of $\Gamma$. So, the existence of the walk $w$ implies that $v'$ is also a root of $\Gamma$, and hence $[g',v']$ is a vertex of $\ol \Gamma_{r}$. Furthermore, the two vertices $[g, v]$ and $[g',v']$ have to be in the same connected component of $\ol \Gamma_r$, and hence $[g', v']$ is a vertex of $\ol \Gamma_{r,i}$, which is a contradiction. This completes the proof. 
\end{proof}

Lemma~\ref{lem:feeluncomfortable} then implies the following fact: if $\ol \Gamma_j$ is rooted, then the root set of $\ol \Gamma_j$ {\it has to} be the vertex set of $\ol \Gamma_{r,i}$. Hence, if $l = |\cal{I}_j| > 1$,  then $\ol \Gamma_j$ cannot be rooted because otherwise, the root set of $\ol \Gamma_{j}$ has to coincide with the vertex set of each  $\ol \Gamma_{r,i}$, for $i\in \cal{I}_j$, which is a contradiction. We have thus proved that if $G^*_i$ is a proper subgraph of $G_i$ for some (and hence all) $v_i \in V_r$, then each connected component of $\ol \Gamma$ is only weakly connected, but not rooted. 

\vspace{3pt}
\noindent
{\bf Case II}. We now assume that  $l = 1$, and prove that $\ol \Gamma_j$ is rooted. In this case, since $k^* = k$, and hence $\cal{I}_j$ is a singleton, we can assume, without loss of generality, that each $\ol \Gamma_{r,j} = (\ol V_{r,j}, \ol E_{r, j})$, for $j = 1,\ldots, k$, is a subgraph of $\ol \Gamma_j$. We establish the following fact:

\begin{lem}\label{lem:feelmuchbetter}
If $l = 1$, then each $\ol \Gamma_j$, for $j = 1,\ldots, k$ is rooted, with $\ol V_{r,j}$ the root set. 
\end{lem}   

\begin{proof}
From Lemma~\ref{lem:feeluncomfortable}, it suffices to show that for each vertex $[g, v]$ of $\ol \Gamma_j$, there is a walk of $\ol \Gamma_j$ from $[g, v]$ to a vertex of $\ol \Gamma_{r, j}$. Let $v_r$ be a root of $\Gamma$, and  $w$ be a walk from $v$ to $v_r$. Then, from the first part of Lemma~\ref{lem:couldbeuseful}, there is a walk of $\ol \Gamma_{j}$ from $[g, v]$ to  $[g \cdot f(w), v_r] $, which is a vertex of both $\ol\Gamma_j$ and $\ol \Gamma_{r}$. Using the fact that $\ol \Gamma_{r, j}$ is the only connected component of $\ol \Gamma_r$ that is contained in $\ol \Gamma_j$, we conclude that  $[g\cdot f(w), v_r] $ is a vertex of $\ol \Gamma_{r, j}$. 
This completes the proof. 
\end{proof}

Combining the results derived in the two cases above,  we  establish Theorem~\ref{thm:coveringgraph}.


\section{The $G$-Clustering Dynamics}


\subsection{Exponential convergence and the adapted partition}

In this section, we investigate the class of $G$-clustering dynamics, for $G$ a point group, and establish relevant properties associated with it.   
Let $(\Gamma, \rho)$ be a $G$-voltage graph, and let $$\theta_{ij}:= \rho(e_{ij}), \hspace{10pt} \forall\, e_{ij} \in E.$$  
We recall that 
a { $G$-clustering dynamics} of a configuration  $p = (x_1, \ldots, x_N)$ is described by the following equation: 
\begin{equation*}
 \dot x_i = \sum_{v_j\in \cal{N}^-(v_i) }a_{ij} \, (\theta_{ij} \, x_j - x_i), \hspace{10pt} \forall i = 1,\ldots, N, 
\end{equation*}
where the $a_{ij}$'s are positive constants. We first establish the following theorem: 


   


\begin{theorem}\label{thm:clustering}
Let  $(\Gamma, \rho)$ be a rooted voltage graph, with the voltage group $G$ a point group in dimension~$k$. 
Let $\{G_i\}_{v_i \in V}$ (resp. $\{G^*_i\}_{v_i \in V}$) be the local groups (resp. directed local groups) of $(\Gamma, \rho)$.  
Suppose that $G_i = G^*_i$ for some (and hence any) root $v_i$ of $\Gamma$; then, for any initial condition $p(0)\in P$, the trajectory $p(t)$, generated by system~\eqref{eq:clusteringdynamics}, converges exponentially fast to a configuration $p^* = (x^*_1, \ldots, x^*_n)$ which satisfies the following two properties:
\begin{enumerate}
\item For each $e_{ij} \in E$, we have $x^*_i = \theta_{ij}\,  x^*_j$. In particular, $\|x^*_1\| = \ldots = \|x^*_N\|$. 

\item For each $v_i\in V$, we have $\theta\,  x^*_i = x^*_i$  for all $\theta \in G_i$.  
\end{enumerate}\,
\end{theorem}

\begin{Remark} Note that if $(\Gamma,\rho)$ is strongly connected, then from Theorem~\ref{thm:coveringgraph}, we have that $G_i = G^*_i$ for all $v_i \in V$.  
Hence, if  the $x_i$'s are scalars and $G = \{1,-1\}$ (and hence $(\Gamma, \rho)$ is a signed graph), then Theorem~\ref{thm:clustering} implies the following fact (Theorem~1 in~\cite{altafini2013consensus}): 
\begin{enumerate}
\item If the signed graph $(\Gamma, \rho)$ is structurally balanced, then for any initial condition $p(0)\in P$, the trajectory $p(t)$, generated by system~\eqref{eq:clusteringdynamics}, converges exponentially fast to a configuration $$p^*= (x^*_1,\ldots, x^*_N)\in P$$ with
$
|x^*_1| = \ldots = |x^*_N|$.
   
\item If the signed graph $(\Gamma, \rho)$ is structurally unbalanced, then for any initial condition $p(0)\in P$, the trajectory $p(t)$, generated by system~\eqref{eq:clusteringdynamics}, 
converges exponentially fast to $\0\in P$.  
\end{enumerate}
The first part directly follows from item~1 of Theorem~\ref{thm:clustering}. For the second part, since $(\Gamma, \rho)$ is structurally unbalanced, $G_i = G = \{1,-1\}$ for some (and hence any) $v_i \in V$. So, from item~2 of Theorem~\ref{thm:clustering}, we have $x^*_i = - x^*_i$, which implies that $x^*_i = 0$ for all $v_i \in V$.  
\end{Remark}

\begin{proof}[Proof of Theorem~\ref{thm:clustering}]
The proof relies on the  construction of an augmented consensus process: 
First, note that for a pair $(\theta, v_i)$ in $G\times V$, the dynamics of $\theta\, x_i$ is given by    
$$
 \theta \,\dot x_i = \sum_{v_j\in \cal{N}^-(v_i)} a_{ij}  ( \theta \,\theta_{ij}\, x_j - \theta\, x_i ).
$$
Hence, if we let $y_{[\theta, v_i]} := \theta \, x_i$,  
then, from Definition~\ref{def:coveringgraph},  the dynamics of $y_{[\theta\, v_i]}$, for $ [\theta, v_i] \in G\times V$, are given by
\begin{equation}\label{eq:augmentedsystem}
\dot y_{[\theta, v_i]} = \sum_{[\theta',v_j]} a_{ij} \( y_{[\theta', v_j]} - y_{[\theta, v_i]}  \),
\end{equation}
where the summation is over all out-neighbors of $[\theta, v_i]$.  
We thus recognize that system~\eqref{eq:augmentedsystem} is a standard consensus process, with the derived graph $\ol \Gamma = (\ol V, \ol E)$ of $(\Gamma, \rho)$ being the underlying network topology. 

Let $k$ be the index of $G_i$ in $G$. Label the connected components of $\ol \Gamma$ as $\ol \Gamma_j = (\ol V_j, \ol E_j)$, for $j = 1,\ldots, k$. Since $G_i = G^*_i$ for any root $v_i$ of $\Gamma$, we know from part~3 of Theorem~\ref{thm:coveringgraph} that each $\ol \Gamma_j$, for $j = 1,\ldots, k$, is rooted.    
Thus, given the initial conditions~$y_{[\theta, v]}(0)$, for $[\theta, v]\in \ol V$,  it is known from~\cite{moreau2004stability} that for each connected component $\ol \Gamma_j$,  there exists a point $y^*_j\in \R^k$ such that along the evolution of the dynamics~\eqref{eq:augmentedsystem}, we have
\begin{equation}\label{eq:consensuscondition}
\lim_{t\to \infty}y_{[\theta, v_i]}(t) = y^*_j, \hspace{10pt} \forall \, [\theta, v_i]\in \ol V_j,
\end{equation}
and the convergence is exponentially fast. The convergence of the $y$-system~\eqref{eq:augmentedsystem} implies the convergence of the $x$-system~\eqref{eq:clusteringdynamics}. Indeed, choose a vertex $v_i$ of $\Gamma$;  without loss of generality, we assume that $[I, v_i]$, for $I$ the identity matrix in $\O(k)$,  is a vertex of $\ol \Gamma_1$. Then, from the definition of $y_{[I, v_i]}$, $$\lim_{t\to \infty}y_{[I, v_i]}(t) = \lim_{t\to \infty} x_i(t) =  y^*_1.$$ 
We next show that 
$
x^*_i = \theta_{ij} \, x^*_j 
$ 
for any edge $e_{ij}$ of $\Gamma$.  
First, note that from Definition~\ref{def:coveringgraph}, each  $[\theta_{ij}, v_j]$, for $v_j\in \cal{N}^-(v_i)$ is an out-neighbor of $[I, v_i]$. Thus, from~\eqref{eq:consensuscondition}, we obtain
$$\lim_{t\to \infty}y_{[I, v_i]}(t) =  \lim_{t\to \infty}y_{[\theta_{ij}, v_j]}(t),$$
which implies that 
$
x^*_i = \theta_{ij} \, x^*_j
$.  
It thus follows that $\|x^*_i\| = \|x^*_j\|$ for any edge $e_{ij}$ of  $\Gamma$. Using the fact that $\Gamma$ is connected, we obtain
$
\|x^*_1\| = \ldots = \|x^*_N\| 
$. 

It remains to show that for each vertex  $v_i \in V$, we have
$
\theta\,  x^*_i = x^*_i$ for all $\theta \in G_i$. 
Because of~\eqref{eq:consensuscondition}, it suffices to show that the two vertices $[I,v_i]$ and $[\theta, v_i]$ of $\ol \Gamma$ belong to the same connected component. This holds because first, by the definition of a local group, there exists a closed semi-walk $w\in SW(v_i,v_i)$ such that $f(w) = \theta$; then, from the first part of Lemma~\ref{lem:couldbeuseful}, there is a semi-walk from 
$[I, v_i]$ to $[\theta, v_i]$ in $\ol \Gamma$.  
This completes the proof. 
\end{proof}

\begin{Remark}
We note here that this ``lifting approach" by lifting a $G$-clustering dynamics to the corresponding $y$-system~\eqref{eq:augmentedsystem} was first proposed by Hendrickx in~\cite{hendrickx2014lifting} for studying the Altafini's model, with $(\Gamma, \rho)$ a strongly connected signed graph.  The proof of Theorem~\ref{thm:clustering} thus generalizes this method so that the ``lifting approach'' can be now applied to a $G$-clustering dynamics, with the corresponding $G$-voltage graph $(\Gamma, \rho)$ satisfying the assumptions of Theorem~\ref{thm:clustering}.    
\end{Remark}

With Theorem~\ref{thm:clustering} at hand, we formalize below in a corollary the following fact: along the evolution of dynamics~\eqref{eq:clusteringdynamics}, the $N$ agents are partitioned into multiple  clusters, with each cluster of agents converging to the same point in $\R^k$. We first have the following definition:

\begin{Definition}[Adapted partition]\label{def:adaptedpartition}
Let $(\Gamma, \rho)$ be a weakly connected voltage graph, with $V$ the vertex set of $\Gamma$. A partition
$
V = \sqcup^m_{l = 1} V_l
$ 
is a {\bf ${(\Gamma, \rho)}$-adapted partition} is defined as follows: two vertices $v_i$ and $v_j$ belong to the same subset if there is a semi-walk $w$ of $\Gamma$ from $v_i$ to $v_j$ such that $f(w) = \1$. 
\end{Definition}

\noindent
Note that a $(\Gamma, \rho)$-adapted partition of $V$ is unique; indeed, the defining condition above establishes an equivalence relation on the set of vertices $V$. 

Recall that for two vertices $v_i$ and $v_j$,  $\net(v_i,v_j)$ is a subset of $G$ given by 
$\net(v_i,v_j) = \{ f(w) \mid w\in SW(v_i,v_j) \}$. Also, recall that 
$
\net(v_i,V) = \cup_{v_j \in V} \net(v_i,v_j)
$. 
We have shown in Proposition~\ref{pro:necessaryandsufficientcondfornond} that $\net(v_i,v_j) = G_i \cdot f(w)$ for $w$ a semi-walk from $v_i$ to $v_j$, and hence $|G_i|$ divides $|\net(v_i,V)|$. We further recall that a voltage graph $(\Gamma, \rho)$ is said to be nondegenerate if $\net(v_i,V) = G$. We now establish the following result as a corollary to Theorem~\ref{thm:clustering}: 

\begin{cor}\label{cor:clustering}
Let $(\Gamma, \rho)$ be a rooted voltage graph that satisfies the assumptions of Theorem~\ref{thm:clustering}. Let $V$ be the set of vertices of $\Gamma$, and $V = \sqcup^m_{l = 1} V_l$ be the  $(\Gamma, \rho)$-adapted partition, with $m$ the number subsets of the partition. Then, the following hold: 
\begin{enumerate}
\item Fix a vertex~$v_i$ of $\Gamma$; then, two vertices~$v_j$ and~$v_k$ belong to the same subset $V_l$, for some $l = 1,\ldots, m$, if and only if 
$$
\net(v_i,v_j) = \net(v_i,v_k).
$$ 
In particular, we have $$m= | \{\net(v_i,V)|/|G_i| \le |G|/|G_i|,$$ 
and the equality holds if and only if $(\Gamma, \rho)$ is nondegenerate. 
\item Let $p(t)$ be a trajectory generated by system~\eqref{eq:clusteringdynamics} that converges to $p^* = (x^*_1,\ldots, x^*_N) $. 
Then, 
$
x^*_i = x^*_j
$ if $v_i$ and $v_j$ belong to the same subset $V_l$ for some $l = 1,\ldots, m$.  
\end{enumerate}\,

\end{cor}

\begin{proof}
We first establish part 1 of the corollary. 
Suppose that $v_j$ to $v_k$ belong to the same subset; then, there is a semi-walk $w_{jk}$ from $v_j$ to $v_k$ such that $f(w_{jk}) = \1$.  We now show that $\net(v_i,v_j) = \net(v_i,v_k)$. 
Choose a semi-walk $w_{ij}$ from $v_i$ to $v_j$; then, by concatenating $w_{ij}$ with $w_{jk}$, we obtain $w_{ik}: =w_{ij}w_{jk}$ as a semi-walk from $v_i$ to $v_k$. Appealing to Proposition~\ref{pro:necessaryandsufficientcondfornond}, we obtain
$$
\begin{array}{lllll}
\net(v_i,v_k)  & = & G_i \cdot f(w_{ik}) & = & G_i \cdot f(w_{ij}) \cdot f(w_{jk}) \\
& = & G_i \cdot f(w_{ij}) & = &  \net(v_i,v_j). 
\end{array}
$$
Conversely, suppose that $\net(v_i,v_j) = \net(v_i,v_k)$;  then, there exist two semi-walks $w_{ij}\in SW(v_i,v_j)$ and $w_{ik}\in SW(v_i,v_k)$ such that $f(w_{ij}) = f(w_{ik})$. Let $w_{jk}:= w^{-1}_{ij} w_{ik}$; then, $w_{jk}$ is a semi-walk from $v_j$ to $v_k$, and moreover, 
$$
f(w_{jk}) = f(w_{ij})^{-1} \cdot f(w_{ik}) = \1,
$$
which implies that $v_j$ and $v_k$ are in the same subset. 

The second part of the corollary directly follows from Theorem~\ref{thm:clustering}; indeed, from part~1 of Theorem~\ref{thm:clustering},  if $w$ is a semi-walk of $\Gamma$ from $v_i$ to $v_j$,  then 
$
x^*_i = f(w)\, x^*_j
$. In particular, if $v_i$ and $v_j$ belong to the same subset $V_l$ for some $l = 1,\ldots, m$, then we can choose $w$ such that $f(w) = \1$, and hence $x^*_i  = x^*_j$,  which completes the proof. 
\end{proof}

\subsection{Simulations}

In this subsection, we illustrate the $G$-clustering dynamics via two concrete examples. We consider a formation of~$8$ agents $x_1,\ldots, x_8\in \R^2$ that evolves according to a $G$-clustering dynamics~\eqref{eq:clusteringdynamics}. For simplicity, all the coefficients $a_{ij}$'s are set to be ones.  
The underlying network topology $\Gamma = (V, E)$ is a strongly connected digraph of~$8$ vertices, illustrated in Fig.~\ref{fig:stronglyconnected}.  
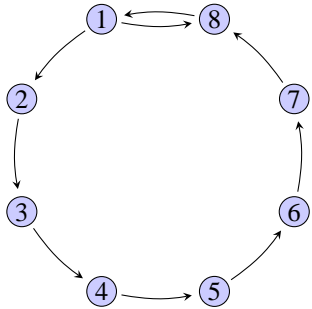
\begin{figure}[h]
\begin{center}
\begin{tikzpicture}[scale = .25, ->,>=stealth,shorten >=1pt,auto,node distance=1.5cm,inner sep=1pt,
  thin,main node/.style={circle,fill=blue!20,draw}]

  \node[main node] (1) {1};
  \node[main node] (8) [right of=1] {8};
  \node[main node] (7) [below right of=8] {7};
  \node[main node] (6) [below of=7] {6};
  \node[main node] (5) [below left of=6] {5};
  \node[main node] (4) [left  of=5] {4};
  \node[main node] (3) [above left  of=4] {3};
  \node[main node] (2) [above  of=3] {2};

  \path[every node/.style={font=\sffamily\small}]
    (1)      edge [bend right=10]  (2)
    (2)      edge [bend right=10]  (3)
    (3)      edge [bend right=10]  (4)
    (4)      edge [bend right=10]  (5)
    (5)      edge [bend right=10]  (6)
    (6)      edge [bend right=10]  (7)
    (7)      edge [bend right=10]  (8)
    (8)      edge [bend right=10]  (1)
    (1)      edge [bend right=10]  (8);
\end{tikzpicture}
\end{center}
\caption{A strongly connected digraph $\Gamma$ of $8$ vertices, with the arrows representing the edges of $\Gamma$.}
\label{fig:stronglyconnected}
\end{figure}

\vspace{3pt} 
\noindent
{\bf Goal}. 
 The goal here is to choose a point group $G$ in dimension~$2$, and to design a map $\rho: E\longrightarrow G$ such that along the dynamics of system~\eqref{eq:clusteringdynamics}  the~$8$ agents are partitioned into~$6$ clusters, and moreover, the associated clustering points form the vertices of a regular hexagon.   
 Specifically, we require that the following two conditions hold for the choices of the point group $G$ and of the map $\rho$: let $p(t) = (x_1(t), \ldots, x_8(t))$ be any trajectory of system~\eqref{eq:clusteringdynamics}; then, $p(t)$ converges to a configuration $p^* = (x^*_1,\ldots, x^*_8)$, with $x^*_i\in \R^2$,  such that the following condition is satisfied:
\begin{enumerate}
\item[A).] $x^*_1 =  \theta^{i - 1}_{\rt, 6} \, x^*_i$ for all $i = 1,\ldots, 8$, where we recall that $\theta_{\rt, 6}$ is a rotation matrix given by
$$\theta_{\rt, 6} = \begin{bmatrix}
\cos(\pi / 3) & -\sin(\pi /3) \\
\sin(\pi / 3) & \cos(\pi /3)
\end{bmatrix}.
$$  
\end{enumerate}\,
Note that from the relation above, we have $x^*_1 = x^*_7$ and $x^*_2 = x^*_8$. 

In the remainder of the subsection, we exhibit two point groups $G$ in dimension~$2$, and correspondingly two different maps $\rho: E\longrightarrow G$, such that the associated $G$-clustering dynamics achieve the goal above.  

\vspace{3pt}

{\bf Example 1}. Let $G$ be a point group isomorphic to $C_6$, i.e., the cyclic group of order~$6$; then, 
$
G = 
\left\langle 
\left\{ \theta_{\rt, 6}
\right\}
 \right\rangle$.  
 Let $(\Gamma,\rho)$ be a voltage graph, with the map $\rho: E \longrightarrow G$ defined as follows: 
\begin{enumerate} 
\item  Let 
$
\rho(e_{i,i+1}) := \theta_{\rt, 6}
$ for $i =1,\ldots, 7$;
\item Let $\rho(e_{8,1}) := \theta_{\rt, 6}^{-1}$;
\item Let $\rho(e_{1,8}) := \theta_{\rt, 6}$.
\end{enumerate}
Then, from Corollary~\ref{cor:stronglyconnectedgraph},  $(\Gamma, \rho)$ is {structurally balanced}. Moreover, a direct computation shows that $(\Gamma, \rho)$ is {nondegenerate}, and the $(\Gamma, \rho)$-adapted partition is given by
\begin{equation}\label{eq:adaptedpartition1}
V = \{1,7\} \cup \{2,8\} \cup \{3\} \cup \ldots \cup \{6\}.
\end{equation}
Let $p(t)$ be a trajectory generated by the $G$-clustering dynamics. Then, from Theorem~\ref{thm:clustering} and Corollary~\ref{cor:clustering}, $p(t)$ converges to a  configuration $p^*= (x^*_1,\ldots, x^*_8)$ which satisfies the condition~A). 

We illustrate the result, via simulation, in Fig.~\ref{fig:ggraph1}. In the simulation, we let the initial condition $p(0) = (x_1(0),\ldots, x_8(0))$ of system~\eqref{eq:clusteringdynamics} be randomly generated: each $x_i(0)$ is a random variable uniformly distributed over the square $[-1,1]\times [-1,1]$ in $\R^2$.  Fig.~\ref{fig:ggraph1} then shows how agents evolve over the plane and converge correspondingly to the vertices of a regular hexagon.

\begin{figure}[h]
\begin{center}
\includegraphics[width=8.5cm]{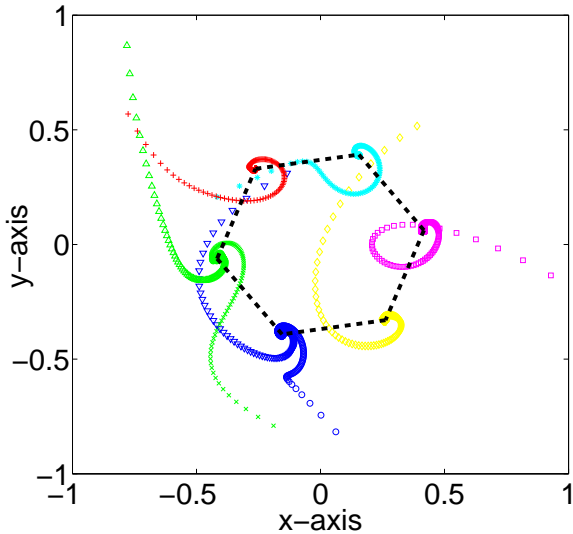}
\caption{Let $(\Gamma, \rho)$ be the $G$-graph defined in Example~1. This figure shows how the agents $x_1,\ldots, x_8$, with randomly chosen initial conditions, evolve along the $G$-clustering dynamics~\eqref{eq:clusteringdynamics}, and converge to the vertices of a regular hexagon. In particular, the two blue (resp. green) curves represent the trajectories of agents $x_1$ and $x_7$ (resp. $x_2$ and $x_8$). We thus see that the pair of agents $(x_1,x_7)$ converges to the same point, and so does the pair $(x_2, x_8)$.    
}\label{fig:ggraph1}
\end{center}
\end{figure}

\vspace{3pt}

{\bf Example 2}. Let $G$ be a point group isomorphic to $D_6$, i.e., the dihedral group of order~$12$. We recall that any such point group is generated by two elements: a rotation matrix $\theta_{\rt, 6}$ and 
$$
\theta_{\rf, v} = 2vv^\top /\|v\|^2 - I, \hspace{10pt} \mbox{for} \hspace{5pt} v \in \R^2 - \{0\} 
$$
which represents the reflection of the line $\{\alpha\, v\mid \alpha\in \R\}$ in $\R^2$. 
Let $(\Gamma, \rho)$ be a voltage graph, with $\rho: E\longrightarrow G$ given by
\begin{enumerate} 
\item Let 
$
\rho(e_{i,i+1}) := \theta_{\rt, 6}
$ for $i =1,\ldots, 7$;
\item Let $\rho(e_{8,1}) := \theta_{\rt, 6}^{-1}$;
\item Let $\rho(e_{1,8}) := \theta_{\rf, v} \theta_{\rt, 6}$.
\end{enumerate}
Note that in this case, the resulting voltage graph $(\Gamma,\rho)$ is structurally unbalanced, but nondegenerate. To see this, let 
$G_1$ be the local group of $(\Gamma, \rho)$ at the vertex~$v_1$. Then, a direct computation yields that $G_1 = \{I, \theta_{\rf, v}\}$, and moreover,  
\begin{equation}\label{eq:ontheflightbacktoshanghai}
\net(v_1, v_i) = G_1 \cdot \theta^{i - 1}_{\rt, 6}, \hspace{10pt} \forall\, i =1,\ldots 8. 
\end{equation}
In particular, $$\{\net(v_1,v_i) \mid i = 1,\ldots, 8 \} = G_1\backslash G,$$
and hence from Proposition~\ref{pro:necessaryandsufficientcondfornond}, the voltage graph $(\Gamma,\rho)$ is nondegenerate. 
Also, note that from~\eqref{eq:ontheflightbacktoshanghai}, 
$$\net(v_1, v_7) = \net(v_1,v_1) \hspace{10pt} \mbox{and} \hspace{10pt} \net(v_1, v_2) = \net(v_1, v_8),$$ 
and moreover, these are the only equalities among the sets $\net(v_1,v_i)$, for $i = 1,\ldots, 8$.  
So, from Corollary~\ref{cor:clustering}, the $(\Gamma, \rho)$-adapted partition of $V$  yields~\eqref{eq:adaptedpartition1}. 

Appealing again to Theorem~\ref{thm:clustering} and Corollary~\ref{cor:clustering}, we conclude that for a trajectory of system~\eqref{eq:clusteringdynamics} that converges to $p^*=(x^*_1,\ldots, x^*_8)$,  the condition~A) is satisfied. 
Furthermore, since $G_1 = \{I, \theta_{\rf,v}\}$,  from part~2 of Theorem~\ref{thm:clustering}, we have
$
\theta_{\rf, v} x^*_1 = x^*_1
$, 
which implies that $x^*_1$ can not be arbitrary, but rather lies on the line $\{\alpha \, v\mid \alpha \in \R\}$. 

We illustrate this result, via simulation, in Fig~\ref{fig:ggraph2}. In the simulation, the initial conditions  $x_i(0)$, for $i = 1,\ldots, 8$, are also randomly generated: each $x_i(0)$ is a  random variable uniformly distributed over $[-1,1]\times [-1,1]$ in $\R^2$.     
The nonzero vector  $v$ is chosen to be $(1,0) \in \R^2$, and hence the matrix $\theta_{\rf, v}$ represents the reflection of the $x$-axis. So, from the analysis, we have that $x^*_1$ lies  on the $x$-axis, which is confirmed by the simulation.  


\begin{figure}[h]
\begin{center}
\includegraphics[width=8.5cm]{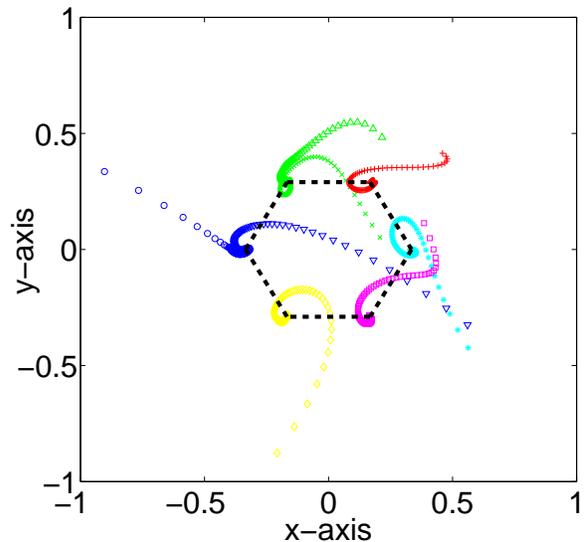}
\caption{Let $(\Gamma, \rho)$ be the $G$-graph defined in Example~2. This figure shows the convergence of the agents $x_1,\ldots, x_8$ to the vertices of a regular hexagon. The pair of agents $(x_1, x_7)$ converges to the same point, and so does the pair $(x_2, x_8)$. Let $x^*_1\in \R^2$ be the point to which the trajectory of agent $x_1$ converges. Then, we see from the figure that $x^*_1$ lies on the $x$-axis. This is a consequence of the fact that $\theta_{\rf,v} x^*_1  = x^*_1$, where $\theta_{\rf,v}$ represents the reflection of the $x$-axis. }\label{fig:ggraph2}
\end{center}
\end{figure}

\section{Conclusions}
A key aspect of modeling dynamics of agents in a large networked system is to design local interaction laws between the individual agents that can lead to some certain desired global behaviors of the ensemble system.  Constructing tractable and flexible models which capture this essential aspect of the network dynamics is a pressing open question. In this paper, we have presented a special class of cluster consensus dynamics, termed $G$-clustering dynamics for $G$ a point group, in which~$N$ autonomous agents interact  with their neighbors to form multiple clusters, with the clustering points satisfying a certain geometric symmetry induced by the associated point group. We have established in Theorem~\ref{thm:clustering} a necessary and sufficient condition for the convergence of a $G$-clustering dynamics. Furthermore, in Corollary~\ref{cor:clustering}, we have counted the number of the associated clusters, and labelled the agents that belong to the same cluster. Toward the analysis of a $G$-clustering dynamics, we have also investigated the underlying $G$-voltage graph and the associated derived graph $\ol \Gamma$ from the perspective of topological graph theory. In particular, we have established, in Subsections II-C, II-D,  and III-B respectively, results about directed local groups of a strongly connected voltage graph, about the existence of nondegenerate and structurally balanced voltage graphs, and about root connectivity of connected components of a derived graph. These results might be of independent interest in topological graph theory.

Future work may focus on the case where the underlying $G$-voltage graph is time-varying. Consider, for example, the map $\rho$ is now a map from $E$ to the power set $2^G$. In other words, each $\rho(e_{ij})$, for $e_{ij} \in E$, is now a subset of $G$. Let $\theta_{ij}(t) \in \rho(e_{ij})$, for $t\ge 0$, be piecewise constant; then, a time-varying $G$-clustering dynamics can be defined as follows: 
\begin{equation}\label{eq:controlsystem}
 \dot x_i = \sum_{v_j\in \cal{N}^-(v_i) }a_{ij} \, (\theta_{ij}(t) \, x_j - x_i), \hspace{10pt} \forall i = 1,\ldots, N, 
\end{equation}
which is a special switching linear system. Establishing stability criterion, such as computing the minimum dwelling time and etc., is in the scope of our future work. We further note that system~\eqref{eq:controlsystem} can be viewed as a bilinear control system if each agent $x_i$ is able to manipulate the values of  $\theta_{ij}(t)$, for $v_j \in \cal{N}^-(v_i)$.  Questions about reachability and controllability can be addressed there.  

 \bibliographystyle{unsrt}
 \bibliography{Ggraph}

\begin{thebibliography}{10}

\bibitem{gross1987topological}
J.L. Gross and T.W. Tucker.
\newblock {\em Topological Graph Theory}.
\newblock Courier Corporation, 1987.

\bibitem{altafini2012dynamics}
C.~Altafini.
\newblock Dynamics of opinion forming in structurally balanced social networks.
\newblock In {\em IEEE Conference on Decision and Control (CDC)}, pages
  5876--5881, 2012.

\bibitem{altafini2013consensus}
C.~Altafini.
\newblock Consensus problems on networks with antagonistic interactions.
\newblock {\em IEEE Trans. on Automatic Control}, 58(4):935--946, 2013.

\bibitem{blondel2009krause}
V.D. Blondel, J.M. Hendrickx, and J.N. Tsitsiklis.
\newblock On {K}rause's multi-agent consensus model with state-dependent
  connectivity.
\newblock {\em IEEE Trans. on Automatic Control}, 54(11):2586--2597, 2009.

\bibitem{yu2010group}
J.~Yu and L.~Wang.
\newblock Group consensus in multi-agent systems with switching topologies and
  communication delays.
\newblock {\em Systems \& Control Letters}, 59(6):340--348, 2010.

\bibitem{xia2011clustering}
W.~Xia and M.~Cao.
\newblock Clustering in diffusively coupled networks.
\newblock {\em Automatica}, 47(11):2395--2405, 2011.

\bibitem{han2013cluster}
Y.~Han, W.~Lu, and T.~Chen.
\newblock Cluster consensus in discrete-time networks of multiagents with
  inter-cluster nonidentical inputs.
\newblock {\em IEEE Trans. on Neural Networks and Learning Systems},
  24(4):566--578, 2013.

\bibitem{shang2013l1}
Y.~Shang.
\newblock L1 group consensus of multi-agent systems with switching topologies
  and stochastic inputs.
\newblock {\em Physics Letters A}, 377(25):1582--1586, 2013.

\bibitem{zaslavsky1989biased}
T.~Zaslavsky.
\newblock Biased graphs. {I}. bias, balance, and gains.
\newblock {\em Journal of Combinatorial Theory, Series B}, 47(1):32--52, 1989.

\bibitem{zaslavsky1991biased}
T.~Zaslavsky.
\newblock Biased graphs. {II}. the three matroids.
\newblock {\em Journal of Combinatorial Theory, Series B}, 51(1):46--72, 1991.

\bibitem{cartwright1956structural}
D.~Cartwright and F.~Harary.
\newblock Structural balance: a generalization of {H}eider's theory.
\newblock {\em Psychological review}, 63(5):277, 1956.

\bibitem{proskurnikov2014consensus}
A.~Proskurnikov, A.~Matveev, and M.~Cao.
\newblock Consensus and polarization in {A}ltafini's model with bidirectional
  time-varying network topologies.
\newblock In {\em IEEE Conference on Decision and Control (CDC)}, pages
  2112--2117, 2014.

\bibitem{xia2015structural}
W.~Xia, M.~Cao, and K.~Johansson.
\newblock Structural balance and opinion separation in trust-mistrust social
  networks.
\newblock {\em IEEE Trans. on Control of Network Systems}, 2015.

\bibitem{Ji2015cdc}
J.~Liu, X.~Chen, T.~Ba{\c s}ar, and M.-A. Belabbas.
\newblock Stability of discrete-time {A}ltafini's model: A graphical approach.
\newblock In {\em IEEE Conference on Decision and Control (CDC)}, 2015.

\bibitem{coxeter2013generators}
H.S.M. Coxeter and W.O.J. Moser.
\newblock {\em Generators and relations for discrete groups}, volume~14.
\newblock {S}pringer {S}cience \& {B}usiness {M}edia, 2013.

\bibitem{rybnikov2005criteria}
K.~Rybnikov and T.~Zaslavsky.
\newblock Criteria for balance in {A}belian gain graphs, with applications to
  piecewise-linear geometry.
\newblock {\em Discrete \& Computational Geometry}, 34(2):251--268, 2005.

\bibitem{moreau2004stability}
L.~Moreau.
\newblock Stability of continuous-time distributed consensus algorithms.
\newblock In {\em IEEE Conference on Decision and Control (CDC)}, volume~4,
  pages 3998--4003, 2004.

\bibitem{hendrickx2014lifting}
J.M. Hendrickx.
\newblock A lifting approach to models of opinion dynamics with antagonisms.
\newblock In {\em IEEE Conference on Decision and Control (CDC)}, pages
  2118--2123, 2014.

\end{thebibliography}


\end{document}